%% file: autosmooth.tex
\documentclass{siamltex}

\usepackage{graphicx}
\usepackage{amssymb}
\usepackage{amsmath}
\usepackage{a4wide}
\usepackage{cite}
\usepackage{verbatim}

\usepackage{tikz}
\usepackage{pgfplots}

\definecolor{darkgreen}{rgb}{0,0.7,0.19}

\newtheorem{remark}{Remark}[section]

\newcommand{\R} {\mathbb{R}}
\newcommand{\C} {\mathbb{C}}
\newcommand{\N} {\mathbb{N}}

\newcommand{\An} {L_n}

\begin{document}
\title{Automatic smoothness detection of the resolvent Krylov subspace method for the approximation of \boldmath{$C_0$}-semigroups}

\author{ Volker Grimm 
        \and Tanja G\"{o}ckler\thanks{
        Karlsruhe Institute of Technology (KIT),
        Institut f\"{u}r Angewandte und Numerische Mathematik,
	D--76128 Karlsruhe,
        Germany, {\tt volker.grimm | tanja.goeckler@kit.edu}.} 
}

\maketitle

\begin{abstract} 
The resolvent Krylov subspace method builds approximations to operator functions $f(A)$ times a vector $v$. For the semigroup and related operator functions, this method is proved to possess the favorable property that the convergence is automatically faster when the vector $v$ is smoother. The user of the method does not need to know the presented theory and alterations of the method are not necessary in order to adapt to the (possibly unknown) smoothness of $v$. The findings are illustrated by numerical experiments. 
\end{abstract}



\section{Introduction}  \label{sec:Introduction}

Let $X$ be some Banach space with norm $\|\cdot\|$. For $t \geq 0$, we consider a $C_0$-semigroup $e^{tA}$, which is generated by $A$, applied to some initial data $v \in X$, or more exactly, 
\begin{equation} \label{semisol}
   u(t)=e^{tA}v, \quad v \in X\,, \quad t \geq 0\,.
\end{equation}
Due to a standard rescaling argument (cf. Section~2.2 on page 60 in \cite{engelnagel00}), it suffices to study bounded semigroups, that is, semigroups satisfying $\|e^{tA}\| \leq N$ for all $t \geq 0$. The object of interest \eqref{semisol}
is just the (mild) solution of the abstract linear evolution equation 
\begin{equation} \label{abstract_ode}
u'(t)=Au(t), \quad u(0) = v\,,\quad t \in [0,\infty)\,,
\end{equation}
whose effective approximation is important in many applications, especially for the numerical solution of semilinear evolution equations by either splitting methods (e.g. \cite{Thalhammer08, McLachlanQuispel02}) or exponential integrators (e.g. \cite{hoacta10}).
In order to approximate the solution \eqref{semisol} of the abstract evolution equation in an efficient and reliable way, one has to use a method which leads to an error reduction that is independent of the norm of the matrix representing the discretized operator $A$ (see \cite{HaiWaodestiff96}). Such error bounds can therefore be designated as grid-independent, since the refinement of the grid in space does not deteriorate the convergence in time (cf. \cite{GG13, TanjaDiss}). 

In the case of a matrix or a bounded operator $A$, the basic importance of an efficient approximation to $e^{tA}$ and possible
methods for this problem are well reflected in ``Nineteen dubious ways to compute the exponential of a matrix'' \cite{Dubious78} by Moler and van Loan. The subsequent finding that the standard Krylov subspace approximation can be used for the approximation of the matrix exponential times a vector, $e^{tA}v$, led to an updated version ``Nineteen dubious ways to compute the exponential of a matrix, twenty-five years later'' with the Krylov subspace method as twentieth method (see \cite{Dubioustwen03}).
Recently, it becomes more and more apparent, that rational Krylov subspace methods constitute a promising twenty-first possibility that is even suitable for matrices with a large norm or unbounded operators. The use of rational Krylov subspaces for the approximation of matrix/operator functions $f(A)$ times $v$ has been studied and promoted, e.g., in \cite{beckermann_guettel12,Beckermann_Reichel09,BoGriHo13,Botchev16,druskin_Zaslavsky12,GG14,GaGriDo96,GG13,grimm_hochbruck08,guettelpole13,HoPaSchuThaWie15,KniDruZas09,lopez_simoncini06,moret15,moretnovati11,moretnovati04,novati11,simoncini06,marlis_jasper}.


In this paper, we will study the approximation of $e^{tA}v$ and products of related operator functions, the so-called $\varphi$-functions, times $v$ in the resolvent Krylov subspace spanned by $(\gamma-A)^{-1}$ and $v$. An efficient approximation of these operator functions is of major importance particularly in the context of exponential integrators.
Our error analysis provides sublinear error bounds for unbounded operators $A$ that translate to error bounds independent of the norm of the discretized operator. That is, the error bounds prove a grid-independent convergence for the discretized problem. Moreover, it turns out that the error reduction correlates with the smoothness of the initial value $v$. A favorable property is that the resolvent Krylov subspace method detects the smoothness of the initial vector by itself and converges the faster the smoother $v$ is. All of this happens automatically, the user of the method does not even need to know the precise smoothness of the initial value.

After this introduction and a motivation in Section~\ref{sec:motivation}, we briefly review a functional calculus in Section~\ref{sec:prelim}. In Section~\ref{sec:approxKry} we prove that any function of the presented functional calculus times a vector can be approximated in the resolvent Krylov subspace spanned by the resolvent and this vector. For the proof of our main results, some smoothing operators are introduced in Section~\ref{sec:smoothing}. In Section~\ref{sec:approxBanach}, the approximation of the semigroup in Banach spaces is considered. Our main theorems can be found in Section~\ref{sec:detect}, where the effect of the smoothness of the vector on the convergence rate for the approximation of the semigroup and related functions in Hilbert spaces is studied. Some numerical illustrations of our results are given in Section~\ref{sec:NumExp}, followed by a conclusion. 


\section{Motivation} \label{sec:motivation}

For a first illustration of this nice feature of the resolvent Krylov subspace method just mentioned above, we consider the one-dimensional Schr\"odinger equation on $L^2(0,2\pi)$
\begin{equation} \label{eq_schroe}
u'(t) = i \tfrac{\partial^2}{\partial x^2} u(t)\,, \quad 
u(0) = u_0
\quad \text{for} \quad 
x \in (0,2\pi)\,,~ t \geq 0\,.
\end{equation}
With $A = \tfrac{\partial^2}{\partial x^2}$, we obtain the abstract equation $u'(t) = iAu(t)$, where the domain of $A$ is the Sobolev space $H^2_\pi(0,2\pi)$ containing all $2\pi$-periodic functions that admit a second order weak derivative. We now discretize \eqref{eq_schroe} by a pseudospectral method. Therefore, we approximate the unknown solution $u$ by a finite linear combination of the basis functions $\phi_k(x) = e^{ikx}$, that is $u(t) \approx \sum_{k=-N/2}^{N/2-1} \psi_k(t) \phi_k$ with $N$ even, and search for coefficients $\psi_k(t)$ such that
\[
\sum_{k=-N/2}^{N/2-1} \psi'_k(t) \phi_k
= \sum_{k=-N/2}^{N/2-1} (-ik^2) \psi_k(t) \phi_k\,.
\]
This ansatz is equivalent to
\begin{equation} \label{schroe_ode}
\Psi'(t) = i A_N \Psi(t)\,, \quad \Psi(0) = \Psi_0
\end{equation}
with solution $\Psi(\tau) = e^{i\tau A_N} \Psi_0$, where the vector $\Psi(t) \in \C^N$ contains the Fourier coefficients $\psi_k(t)$ for $k = -\frac{N}{2}, \ldots, \frac{N}{2}-1$ and the matrix $A_N \in \R^{N \times N}$ is a diagonal matrix with entries $\left(-\frac{N}{2}\right)^2, \left(-\tfrac{N}{2}+1\right)^2, \ldots, \left(\tfrac{N}{2}-1\right)^2$. The discretized initial vector $\Psi_0 = \big(\psi_k(0)\big)$ is given by 
\[
\psi_k(0) = \frac{1}{2\pi} \int_0^{2\pi} u_0(x) e^{-ikx} \,dx\,,
\quad k = -\tfrac{N}{2}, \ldots, \tfrac{N}{2}-1\,.
\]
These coefficients $\psi_k(0)$ can be approximated by a discrete Fourier transform of the discretized function $u_0$. Here, we use the initial data
\[
u^q_0(x) =
\left\{ \begin{array}{l@{\qquad}l}
\left( \tfrac{2}{\pi} \right)^{4q} (x-\pi)^{2q} x^{2q}\,,
& x \in (0, \pi]\,, \\[2ex]
\left( \tfrac{2}{\pi} \right)^{4q} (x-\pi)^{2q} (x-2\pi)^{2q}\,,
& x \in (\pi, 2\pi]\,.
\end{array} \right.
\]
Differentiating this function $2q+1$ times, $\frac{d^{2q+1}}{dx^{2q+1}} u_0^q$ becomes
discontinuous at $x = \pi$ and at $x = 2\pi$, if $u_0^q$ is considered as a $2\pi$-periodic function. So, we have $u_0^q \in \mathcal{D}(A^q)$ but $u_0^q \not\in \mathcal{D}(A^{q+1})$.
By $\Psi_0^q \in \C^N$, we denote the corresponding spectral discretizations of the initial value. The solution $e^{i\tau A_N} \Psi_0^q$ of the discretized initial value problem at time $\tau > 0$ is now approximated in the rational Krylov subspace $\mathcal{K}_n((\gamma-i\tau A_N)^{-1}, \Psi_0^q)$, where $\gamma = 1$.

In Figure~\ref{fig_moti},  the error of the rational Krylov subspace approximation is plotted against the dimension of the Krylov subspace (blue solid lines) for $N = 131072$, $\tau = 0.02$, and smoothness indices $q = 2, 4, 6, 8$. We can observe that $e^{i\tau A_N} \Psi_0^q$ is approximated the better the smoother the continuous initial value $u_0^q$ is, or more exactly, the higher the number $q$ with $u_0^q \in \mathcal{D}(A^q)$ is. Furthermore, we applied for comparison the implicit Euler method to the discretized problem and added the obtained error curves to Figure~\ref{fig_moti} (red dashed lines). 

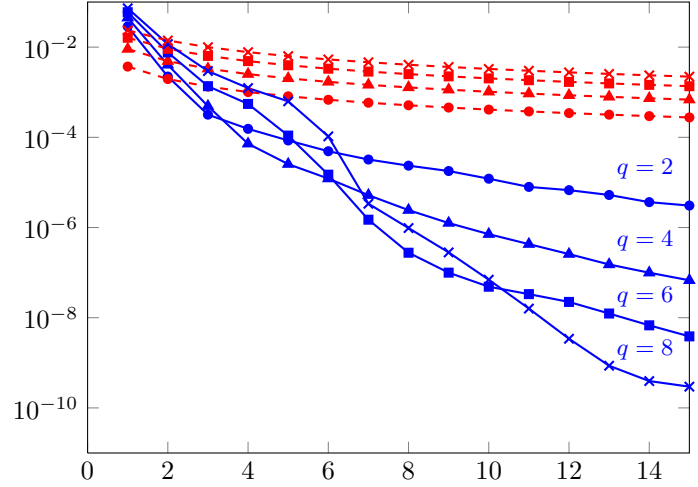
\begin{figure} \centering
\input{./Schroe_17_0p02_20.tikz}
\caption{Plot of the error versus dimension of the Krylov subspace $\mathcal{K}_n((1-i\tau A_N)^{-1}, \Psi_0^q)$ (blue solid lines), and of the implicit Euler method (red dashed lines) for $N = 131072$, $\tau=0.02$ and initial vectors $\Psi_0^q$ resulting from the discretization of $u_0^q \in \mathcal{D}(A^q)$ for $q = 2, 4, 6, 8$ (circle-, triangle-, square-, cross-marked line).}
\label{fig_moti}
\end{figure}

In order to introduce the resolvent Krylov subspace and to get a first idea why this subspace might be a good choice for the approximation of the operator/matrix exponential, we consider  
the implicit Euler scheme, which is, besides the explicit Euler scheme, a standard method to approximate the matrix exponential times a vector. The two methods are based on the relations
\[
\begin{array}{l@{\qquad}l@{\hspace*{0.1cm}}l}
\text{explicit Euler:} & \displaystyle
\lim_{n \to \infty} \left( I + \tfrac{\tau}{n}A \right)^nv &= e^{\tau A}v\,, \\[1ex] 
\text{implicit Euler:} & \displaystyle 
\lim_{n \to \infty} \left( I - \tfrac{\tau}{n}A \right)^{-n}v &= e^{\tau A}v\,.
\end{array}
\]
For matrices with a large norm, the explicit Euler method does not work efficiently. The discretized Schr\"odinger equation is a stiff system of ordinary differential equations. If we increase the number $N$ of basis functions, the norm of the discretization matrix $A_N$ grows. For the discretized problem considered here, the explicit Euler method is therefore not suitable. The situation is even worse for the continuous equation, since the explicit Euler scheme cannot be used unless the initial data is very smooth and lies in $\mathcal{D}(A^\infty) = \cap_{n=1}^\infty \mathcal{D}(A^n)$. The implicit Euler method, however, provides an approximation to the semigroup for all initial vectors $v$ in the associated Banach space~$X$. The resolvent $(I-\frac{\tau}{n}A)^{-1}$ maps $X$ to $\mathcal{D}(A)$ and can thus be seen as a smoothing operator. While the explicit Euler method cannot be applied for initial values $v \not\in \mathcal{D}(A^\infty)$, the implicit Euler method can be proven to possess the convergence rates (see \cite{brenner_thomee79})
\[
\left\| e^{\tau A}v - \left( 1-\tfrac{\tau}{n}A \right)^{-n} v \right\|
\leq
\left\{ \begin{array}{l@{\qquad}l}
C\, \tfrac{\tau}{\sqrt{n}}\, \|A v\|\,, 
& v \in \mathcal{D}(A)\,, \\[2ex]
C\, \tfrac{\tau^2}{n}\, \|A^2 v\|\,, 
& v \in \mathcal{D}(A^2)\,.
\end{array} \right. 
\]
For even smoother data, the implicit Euler method does not converge faster. 
In order to improve both methods, we briefly review the basic idea of Krylov subspace methods. Consider 
$n-1$ steps of the explicit Euler method that can be seen as a product of a polynomial in $A$ and $v$, that is
\[
e^{A}v \approx \bigl( 1+\tfrac{\tau}{n-1}A \bigr)^{n-1}v = a_0 v + a_1 A v + \cdots + a_{n-1} A^{n-1} v = p(A)v\,, 
\quad p \in \mathcal{P}_{n-1}\,,
\]
where $\mathcal{P}_{n-1}$ is the space of polynomials of maximum degree $n-1$. Instead of using a fixed polynomial approximation given by the explicit Euler scheme, it might be better to search a best approximation in the polynomial Krylov subspace
\[
\mathcal{K}_n(A,v)=\mbox{span}\{v, A v, A^2 v, \ldots, A^{n-1} v\}\,.
\]
Using this approximation space for the numerical solution of the discretized Schr\"odinger equation, it turns out that the approximation improves with respect to stability, but a substantial error reduction would just begin after nearly $\|\tau A_N\|$ iteration steps (see \cite{Kry3}), where $\|\tau A_N\|$ becomes large for fine space discretizations. Analogous to the explicit Euler method, the standard polynomial Krylov subspace method is not suitable for a grid-independent approximation of the Schr\"odinger equation.

Instead of applying the implicit Euler method, one can try to find a better approximation of the type
\[
e^A v \approx a_0 v + a_1 (\gamma-A)^{-1} v + \cdots + a_{n-1} (\gamma-A)^{-(n-1)} v = r(A) v\,,
\quad r \in \frac{\mathcal{P}_{n-1}}{(\gamma-\cdot)^{n-1}}\,,
\]
that means, to search a best approximation in the rational Krylov subspace 
\[
  \mathcal{K}_n((\gamma-A)^{-1},v) = \mathrm{span} \{v, (\gamma-A)^{-1}v, (\gamma-A)^{-2}v,\ldots, (\gamma-A)^{-n+1}v \}\,, \quad \gamma > 0\,.
\]
This so-called resolvent Krylov subspace has been proposed by Ruhe in \cite{Ruh84} for eigenvalue computations and is by now a standard technique for this purpose (cf. \cite{tempeigval00}). We will use the resolvent Krylov subspace as approximation space for the approximation of $e^Av$ and related operator functions in the following. Analogous to the implicit Euler method, the approximation based on the resolvent Krylov subspace will be grid-independent but improve on the implicit Euler method with respect to the convergence rate dependent on the smoothness of the vector $v$, as illustrated above (cf. Figure~\ref{fig_moti}).

\section{Preliminaries} \label{sec:prelim}

We briefly review a functional calculus that has been formerly used in \cite{GG13} and \cite{grimm_gugat10}. The Lebesgue space of complex-valued integrable functions defined on $\R$ is denoted by $L^1(\R)$ with norm $\| \cdot \|_1$. By $C(\R)$, we designate the space of continuous functions $f \,:\, \R \to \C$. Moreover, let
\begin{equation} \label{Mplus}
\mathcal{M}_+ = 
\left\{
f \in C(\R) ~ | ~ \mathcal{F} f \in L^1(\R) ~ \mbox{and} ~ 
\supp (\mathcal{F}f) \subseteq [0,\infty)
\right\},
\end{equation}
where $\mathcal{F}f$ is the Fourier transform of $f$ given as
\[
\mathcal{F} f(s) = \frac{1}{2\pi} \int_{-\infty}^\infty e^{-ixs} f(x) \, dx \quad \mbox{for} \quad f \in L^1(\R)\,.
\]
For $f \not \in L^1(\R)$, the Fourier transform is understood in the sense of distributions.
For each function holomorphic in the left half-plane, we denote by $f_{(0)}: \R \rightarrow \C$ the restriction of $f$ to $\mbox{Re}\, z = 0$ so that $f_{(0)}(\xi)=f(i\xi)$, $\xi \in \R$, and we define the algebra
\[
\widetilde{\mathcal{M}} := 
\left\{
f ~ \mbox{holomorphic and bounded for} ~ \mbox{Re} \, z \leq 0 ~ | ~ f_{(0)} 
\in \mathcal{M}_+ \right\}.
\]
Let $A$ generate a bounded strongly continuous semigroup with $\|e^{\tau A}\| \leq N$ on some Banach space $X$. For functions $f \in \widetilde{\mathcal{M}}$, we introduce a functional calculus via
\begin{equation} \label{newcalcdef}
f(A)= \int_0^\infty e^{sA} \, \mathcal{F}f_{(0)}(s) \, ds\,.
\end{equation}
This defines a bounded linear operator $f(A)$ satisfying
$\|f(A)\| \leq N\|\mathcal{F}f_{(0)}\|_1$.
Until we know that the functional calculus is consistent with standard operator functions such as the resolvent and the semigroup, 
we write $\left(f(z) \right)(A)$, when the definition of the operator functions is according to the new calculus \eqref{newcalcdef}, instead of simply $f(A)$. For $f(z)=(z_0-z)^{-k}$ with $\mbox{Re}\, z_0 > 0$ and $k \geq 1$, we have by elementary semigroup theory (cf. Corollary~1.11, pp. 56--57 in \cite{engelnagel00}),
\[
\left( \frac{1}{(z_0-z)^k} \right)(A)  = 
\int_0^\infty e^{sA} e^{-sz_0}\cdot \frac{s^{k-1}}{(k-1)!} \, ds=(z_0-A)^{-k}\,,
\]
that is, the definition via \eqref{newcalcdef} coincides with the definition in terms of the resolvent. Analogously, all rational functions with a smaller degree of the numerator than the denominator and poles in the right complex half-plane are included by our functional calculus so far. 

We will need another extension in order to include the semigroup, i.e., we want that the generator $A$ inserted in the exponential function $e^{tz}$, $t \geq 0$, coincides with the semigroup. Let
\[
\mathcal{M}_0 := \{ f ~ \mbox{holomorphic for} ~ \mbox{Re} \, z \leq 0 ~ | ~ 
\exists \, n \in \N_0: \frac{f(z)}{(1-z)^n} \in \widetilde{M} \}\,.
\]  
For $f \in \mathcal{M}_0$, we set
\begin{equation*} 
f(A):= (1-A)^n \left( \frac{f(z)}{(1-z)^n} \right)(A)\,,
\end{equation*}
where $n$ is such that $\frac{f(z)}{(1-z)^n} \in \widetilde{M}$. Note that the definition does not depend on the choice of $n$ and that the definition results in a closed operator on $X$. Finally, we define the set
\[
\widetilde{\mathcal{M}} \subseteq
\mathcal{M} := 
  \{ f \in \mathcal{M}_0 ~ | ~ f(A): X \rightarrow X ~ \mbox{is bounded} \}
\]
which is sufficient for our purposes. The following lemma can be found as Proposition 1.12 in \cite{Haasethesis}.
\begin{lemma} \label{homomorphy2}
The mapping $f \rightarrow f(A)$ via \eqref{newcalcdef} is a homomorphism of $\mathcal{M}$ into the algebra of bounded linear operators on $X$.
\end{lemma}

We can check, that the semigroup is now included in the extended functional calculus.
\begin{lemma}
For $\tau \geq 0$, we have
\[
  \left( e^{\tau z} \right)(A) = e^{\tau A}\,.
\]
\end{lemma}
\begin{proof}
For $n=1$, one can verify that $\frac{e^{\tau z}}{1-z} \in \widetilde{\mathcal{M}}$. Hence, we have by \eqref{newcalcdef} that
\begin{align*}
  \left( \frac{e^{z\tau}}{1-z} \right)(A) &= \int_0^\infty e^{sA}{\bf 1}_{[\tau,\infty)}(s) e^{\tau-s}\, ds    = \int_{\tau}^\infty e^{sA}e^{\tau-s}\, ds \\
  &= \int_0^\infty e^{(s+\tau)A}e^{-s}\,ds = e^{\tau A} \int_0^\infty e^{sA}e^{-s}\,ds 
  = e^{\tau A}(1-A)^{-1}\,.
\end{align*}
Finally, we conclude
\[
  (1-A) \left( \frac{e^{\tau z}}{1-z} \right)(A)=(1-A)e^{\tau A}(1-A)^{-1}=e^{\tau A}
\]
which proves the assertion.
\end{proof}

For all functions relevant to our discussion, the functional calculus  \eqref{newcalcdef} coincides with the definitions in semigroup theory. From now on, we therefore do not use different notations and simply write $f(A)$ for a function $f$ of an operator $A$ with respect to \eqref{newcalcdef} . We will also need the following lemma of Brenner and Thom\'{e}e (cf. Lemma~4 in \cite{brenner_thomee79}), whose proof extends to our case. 
\begin{lemma} \label{btlemma4}
For $f,g \in \mathcal{M}$ with $f(z)=z^lg(z)$ for some $l>0$ and $\mbox{Re}\,z \leq 0$, we have
\[
f(A)v=g(A)A^lv \quad \text{for} \quad v\in\mathcal{D}(A^l)\,.
\]
\end{lemma}

\section{Approximation in the resolvent Krylov subspace} \label{sec:approxKry}

Here and in the following, we always consider bounded semigroups with generator $A$ on some Banach space $X$ which satisfy $\|e^{tA}\| \leq N$. For bounded semigroups, it is well-known that the right complex half-plane belongs to the resolvent set of the generator $A$ (e.g. Theorem 1.10 on page 55 in \cite{engelnagel00}) which guarantees that the resolvent $(\gamma - A)^{-1}$ exists for all $\gamma > 0$. 

We are interested in the approximation of operator functions, especially the semigroup, times a vector $v \in X$ in the resolvent Krylov space
\begin{equation} \label{reskrysp}
   \mathcal{K}_n((\gamma-A)^{-1},v) := \mathrm{span} \{v, (\gamma-A)^{-1}v, (\gamma-A)^{-2}v,\ldots, (\gamma-A)^{-n+1}v \}\,, \quad \gamma > 0\,.
\end{equation}
For $n = 1, 2, 3, \ldots$, these spaces form a nested sequence of subspaces. If there exists an index $n_0$ for which $\mathcal{K}_{n_0}((\gamma-A)^{-1},v)$ is invariant under $(\gamma-A)^{-1}$, we have $\mathcal{K}_{n_0}((\gamma-A)^{-1},v) = \mathcal{K}_k((\gamma-A)^{-1},v)$ for all $k \geq n_0$.
For a Banach space $X$ of finite dimension, this always happens. At the latest, when $n$ reaches the dimension of $X$. For a Banach space of infinite dimension, this might happen or it might not. In most cases, the spaces build an infinite series of nested spaces that are different. We therefore first discuss the natural question, whether all functions of our functional calculus can be approximated to an arbitrary precision in the space \eqref{reskrysp} when $n$ tends to infinity. For this purpose, we define the maximal resolvent Krylov subspace. 
\begin{definition} The maximal resolvent Krylov subspace for a given vector $v \in X$ and a fixed $\gamma > 0$ is given as the space
\begin{equation} \label{maxresKry}
  \mathcal{K}_\infty((\gamma-A)^{-1},v) := \mathrm{span} \{v, (\gamma-A)^{-1}v, (\gamma-A)^{-2}v,\ldots \}\,.
\end{equation}
We also need the closure of this space that we designate by 
$\overline{\mathcal{K}_\infty((\gamma-A)^{-1},v)} \subseteq X$.
\end{definition}

The following theorem states that all functions that are defined for $A$ via \eqref{newcalcdef} times $v$ are in the closure of the maximal resolvent Krylov subspace \eqref{maxresKry}, that is, $f(A)v$ can be approximated in the Krylov subspace \eqref{maxresKry} to any desired precision. Since the span designates all finite linear combinations, this also means that all functions in our functional calculus can be approximated in the space \eqref{reskrysp} to any arbitrary precision, if we let $n$ go to infinity.
\begin{theorem} \label{fAinclK}
For all $v \in X$ and all functions $f \in \mathcal{M}$, we have
\[
   f(A)v \in \overline{\mathcal{K}_\infty((\gamma-A)^{-1},v)}\,.
\]
\end{theorem}
\begin{proof}
If we define
\[
Y:=\overline{\mathcal{K}_\infty((\gamma-A)^{-1},v)}\,, 
\]
then $Y$ is an invariant subspace of $(\mu-A)^{-1}$ for all $\mu > 0$ (cf. proof of Theorem~4.6.1 in \cite{Mikl98}). Hence, we have 
\[
 (\mu -A)^{-1}y \in Y \quad \mathrm{for~all} \quad y \in Y\,,~ \mu > 0\,.  
\]
Theorem~4.6.1 in \cite{Mikl98} now states that $Y$ is an invariant subspace of our semigroup $e^{tA}$, $t \geq 0$, and $A$. Furthermore, the restriction $\left. e^{tA} \right|_{Y}$ of the semigroup $e^{tA}$ to $Y$ is again a semigroup with generator $\left. A \right|_{Y}$ and $\left. A \right|_{Y}y=Ay$ for all $y \in \mathcal{D}(A) \cap Y =: \mathcal{D}(\left. A \right|_Y)$. For $f \in \widetilde{\mathcal{M}} \subseteq \mathcal{M}$, we thus find
\[
f(A)y=\int_0^\infty e^{sA}y \,\mathcal{F}f_{(0)}(s)\,ds = \int_0^\infty e^{s\left. A \right|_Y} y\, \mathcal{F}f_{(0)}(s)\,ds
= f\left( \left. A \right|_Y \right)y \in Y \quad \mbox{for all} \quad y \in Y.
\]
Since $v \in Y$, we obtain $f(A)v\in Y$. Now we proceed with the case of functions belonging to $\mathcal{M}$. By the definition of $\mathcal{M}$, we have for $f \in \mathcal{M}$ and all $y \in Y$ that
\[
  f(A)y=(1-A)^n \left( \frac{f(z)}{(1-z)^n} \right)(A)y = (1-A)^n g(A)y \quad \mbox{with} \quad g(z)=\frac{f(z)}{(1-z)^n}\,, \quad g \in \widetilde{\mathcal{M}}\,,
\] 
where $n$ has been chosen appropriately. Because of $Y$ being $A$-invariant, we obtain
\[
(\gamma-A)^{l}y=\left(\gamma-\left. A \right|_Y \right)^{l}y \quad \mbox{for all} \quad y \in \mathcal{D}\left( \left( \left. A\right|_Y \right)^l \right), \quad l \in \N\,.
\] 
By the first part of the proof, since $g \in \widetilde{\mathcal{M}}$, we can conclude for $y \in Y$ that
\[
f(A)y = (1-A)^n g(A)y = (1-A)^n g\left( \left. A \right|_Y \right)y= \left( 1-\left. A \right|_Y \right)^n g\left( \left. A \right|_Y \right)y
= f\left(\left. A\right|_Y \right)y \in Y.
\]
Again, due to $v \in Y$, the statement $f(A)v \in Y$ follows.
\end{proof}

In the case that an index $n_0$ exists for which the resolvent Krylov subspace is invariant, we obtain from Theorem~\ref{fAinclK}, that 
\[
   f(A)v \in \mathcal{K}_{n_0}((\gamma-A)^{-1},v)=\overline{\mathcal{K}_\infty((\gamma-A)^{-1},v)}\,.
\]
Thus, $f(A)v$ can  be represented exactly in the finite-dimensional space \eqref{reskrysp} with index $n = n_0$.

We also study subspaces of $X$ of the type 
\begin{equation} \label{smoothiv}
\mathcal{K}_\infty((\gamma-A)^{-1},(\gamma-A)^{-q}v), \qquad q = 1,2,3,\ldots
\end{equation}
with a smoothed initial vector. These spaces are usually different. For example, if $v \in X \backslash \mathcal{D}(A)$ holds true, then $v$ is in the space \eqref{maxresKry}, but $v$ is not in any of the spaces \eqref{smoothiv}, which are all subsets of $\mathcal{D}(A)$. An intriguing fact is that the closures of the spaces \eqref{smoothiv} are identical and coincide with the closure of \eqref{maxresKry}. Hence, from a numerical analyst's point of view, if $w \in X$ can be approximated to an arbitrary precision in any of the spaces \eqref{maxresKry} or \eqref{smoothiv}, then $w$ can be approximated in all spaces to an arbitrary precision.
\begin{lemma} For every $q =1,2,3,\ldots$ and $v \in X$, we have
\[
\overline{\mathcal{K}_\infty((\gamma-A)^{-1},(\gamma-A)^{-q}v)} = \overline{\mathcal{K}_\infty((\gamma-A)^{-1},v)}\,. 
\]
\end{lemma}
\begin{proof}
If we set
\begin{equation} 
Y:=\overline{\mathcal{K}_\infty((\gamma-A)^{-1},(\gamma-A)^{-1}v)}\,,
\end{equation}
then we have, analogously to the previous proof, that
\[
 (\mu -A)^{-1}y \in Y \quad \mathrm{for~all} \quad y \in Y\,,~ \mu > 0\,.  
\]
Obviously, $(\gamma-A)^{-1}v \in Y$ and hence $(\gamma-\mu)(\mu-A)^{-1}(\gamma-A)^{-1}v \in Y$. By the resolvent equation (cf. (1.2) on page 239 in \cite{engelnagel00}), it follows
\[
(\mu-A)^{-1}v = (\gamma-A)^{-1}v+(\gamma-\mu)(\mu-A)^{-1}(\gamma-A)^{-1}v \in Y\,.
\]
Since $\mu > 0$ has been arbitrarily chosen, we have $\mu(\mu-A)^{-1}v \in Y$ for all $\mu > 0$. Due to the well-known fact that
\[
  \lim_{\mu \rightarrow \infty} \mu(\mu-A)^{-1}v=v
\]
(e.g. Lemma 3.4, p.\,73 in \cite{engelnagel00}), we find $v \in Y$, since $\mu(\mu-A)^{-1}v \in Y$ and $Y$ is closed. This immediately shows our assertion for $q=1$. The statement for $q > 1$ now follows by induction.
\end{proof}

\section{Smoothing operators with range in the resolvent Krylov subspace} \label{sec:smoothing}

We study in this section, how well an initial vector $v \in \mathcal{D}(A^q)$ can be approximated in different resolvent Krylov subspaces. These bounds will be necessary to prove our main theorems. The constants occurring in the following bounds and estimates will always be generic constants denoted by $C(param_1, param_2, \ldots)$, where the terms in brackets indicate the parameters on which $C$ depends.

The first lemma provides an error estimate for the approximation of $v \in \mathcal{D}(A^q)$ in the special resolvent Krylov subspace $\mathcal{K}_q((\sqrt{n}-A)^{-1}, (\sqrt{n}-A)^{-q}v)$.

\begin{lemma}\label{smoothop} There are bounded operators $H_{n,q}$ of the form
\[
H_{n,q} = \sum_{k=q}^{2q-1} h_k^q\left( \frac{\sqrt{n}}{\sqrt{n}-A} \right)^k, \quad
h_k^q= \binom{2q-1}{k} \sum_{l=0}^{k-q} \binom{k}{l}(-1)^l=\binom{2q-1}{k}\binom{k-1}{k-q}(-1)^{k-q}\,,
\]
such that, for all $v \in \mathcal{D}(A^q)$, we have
\[
\|H_{n,q}v-v\| \leq \frac{C(q,N)}{n^{\frac{q}{2}}} \|A^qv\|
\quad \text{and} \quad 
\|H_{n,q}\| \leq C(q,N)\,,
\]
where the constants depend only on $q$ and the bound $N$ of the bounded semigroup, but not on $n$. 
\end{lemma}
\begin{proof}
The coefficients $h_k^q$, $k = q, \ldots, 2q-1$, are chosen such that
\[
  g(z) := \frac{1-\sum_{k=q}^{2q-1} h_k^q (1-z)^{-k}}{z^q}
\]
can be continued to a holomorphic function on $\C\backslash \{1\}$. One can check that $g(z)$ is a linear combination of $(1-z)^{-q}, (1-z)^{-(q+1)}, \ldots, (1-z)^{-(2q-1)}$. For any generator $B$ of a bounded semigroup $e^{tB}$ with $\|e^{tB}\| \leq N$, we have $\|(1-B)^{-1}\| \leq N$ (cf. Theorem 1.10 on page 55 in \cite{engelnagel00}) and thus
\[
\left\| \frac{1-\sum_{k=q}^{2q-1} h_k^q (1-B)^{-k}}{B^q}\right\| \leq C(q,N)\,.
\]
We will now use this estimate for $B = \frac{1}{\sqrt{n}} A$, which generates a semigroup satisfying $\|e^{\frac{t}{\sqrt{n}} A}\| \leq N$ to bound the difference $\| H_{n,q}v-v\|$.
Since $g \in \mathcal{M}$, $1 \in \mathcal{M}$ and $\sum_{k=q}^{2q-1} h_k^q (1-z)^{-k} \in \mathcal{M}$ are functions that belong to our extended functional calculus, we obtain according to Lemma~\ref{btlemma4}
\begin{eqnarray*}
\left\| H_{n,q}v-v\right\| &=& 
\left\| 
   \left( 
      1-\sum_{k=q}^{2q-1} h_k^q
	\left(\frac{\sqrt{n}}{\sqrt{n}-A}\right)^k 
   \right)v 
\right\| \\[1ex]
&\leq& 
\left\|
\frac{
      1-\sum_{k=q}^{2q-1} h_k^q
      \left(1-\frac{1}{\sqrt{n}}A\right)^{-k} 
}{
\left( \frac{1}{\sqrt{n}} A\right)^q
}
\right\|  \cdot \left\|\left( \tfrac{1}{\sqrt{n}} A\right)^qv\right\|
\leq  \frac{C(q,N)}{n^{\frac{q}{2}}}\|A^qv\|\,.
\end{eqnarray*}
The bound on $H_{n,q}$ follows, since $\|(\sqrt{n}-A)^{-1}\|\leq N/\sqrt{n}$ (e.g. Theorem~1.10 on page 55 in \cite{engelnagel00}) and therefore $\| \sqrt{n} (\sqrt{n}-A)^{-1} \| \leq N$.
\end{proof}

In the next lemma, we study the best approximation $W_n$ to $\sqrt{n}(\sqrt{n}-A)^{-1}$ in the resolvent subspace
\begin{equation} \label{reskryop}
\mathcal{R}_n(\gamma,A) := \mathrm{span} \left\{ (\gamma-A)^{-1}, (\gamma-A)^{-2}, \ldots, (\gamma-A)^{-n}\right\}, \quad \gamma > 0\,.
\end{equation}

\begin{lemma} \label{smoothapprox}
For any $p \in \N$ and any $n \in \N$, there are operators $W_n$ of the form 
\[
W_n = \sum_{k=1}^n w_k^n \frac{1}{(\gamma-A)^k} 
\in \mathcal{R}_n(\gamma,A)
\quad \text{with} \quad w_k^n \in \C
\]
such that
\[
\left\| \frac{\sqrt{n}}{\sqrt{n}-A}-W_n \right\| 
\leq \frac{C(\gamma,p,N)}{n^{\frac{p}{2}}}
\quad \text{and} \quad \|W_n\| \leq C(\gamma,p,N)
\]
with $n$-independent generic constants $C(\gamma,p,N)$.
\end{lemma}

\begin{proof}
We estimate the best approximation in the space \eqref{reskryop}. This best approximation exists, since the space is finite. By standard calculations with the resolvent, we obtain
\[
\mathcal{R}_m(\gamma,A) = \mathrm{span} \left\{ (\gamma-A)^{-1}, A(\gamma-A)^{-2}, \ldots, A^{m-1}(\gamma-A)^{-m} \right\},
\quad \gamma > 0\,.
\]
We now proceed analogously to the proof of Theorem~4.1 in \cite{ratkryphi11}. By using
\begin{align*}
  \frac{\sqrt{n}}{\sqrt{n}-A} &= \sqrt{n} \int_0^\infty e^{sA} e^{-\sqrt{n}s}\,ds\,, \\[1ex]
  \frac{A^{m-1}}{(\gamma-A)^m} &= \int_0^\infty e^{sA} e^{-\gamma s}(-1)^{m-1}L_{m-1}(\gamma s)\,ds
\end{align*}
with the $m$-th Laguerre polynomial
\[
L_m(x)=\sum_{k=0}^m \binom{m}{k} \frac{(-1)^k}{k!}\,x^k,
\]
it follows
\[
\left\| \frac{\sqrt{n}}{\sqrt{n}-A}-\sum_{k=1}^m a_k \frac{A^{k-1}}{(\gamma-A)^k} \right\| 
\leq \frac{N}{\gamma} \cdot E_m \big( e^s \mathcal{F}f_{(0)}(\tfrac{s}{\gamma}) \big)\,, 
\quad e^s\mathcal{F}f_{(0)}(\tfrac{s}{\gamma})=\sqrt{n}e^{(1-\frac{\sqrt{n}}{\gamma})s},
\] 
where $\mathcal{F}f_{(0)}$ is the Fourier transform of $f(z) = \frac{\sqrt{n}}{\sqrt{n}-z}$ restricted to $\text{Re}~z = 0$. It remains to estimate
\[
E_m \big( e^s \mathcal{F}f_{(0)}(\tfrac{s}{\gamma}) \big)
= \inf_{a_1,\ldots,a_m} 
\big\| e^{-s} \big( \sqrt{n}e^{(1-\frac{\sqrt{n}}{\gamma})s} - \sum_{k=1}^m a_k (-1)^{k-1} L_{k-1}(s) \big) \big\|_1\,.  
\]
Let $L_{\omega_0}^1(\R)$ be the space of Lebesgue integrable functions with respect to the weight function $\omega_0$. Moreover, we equip the space $W_r^1(\omega_0) := \left\{ g \in L_{\omega_0}^1(\R)\,:\, \left\| g^{(r)}\varphi^r\omega_0\right\|_1 < \infty \right\}$ with the norm
\[
\|g\|_{W^1_r(\omega_0)} := \| g\omega_0\|_1 + \|g^{(r)}\varphi^r\omega_0 \|_1\,,
\quad \text{where} \quad \varphi(s) = \sqrt{s}\,, \quad \omega_0(s) = e^{-s}.
\]
According to \cite{JoNguy89}, there exists for any $r \geq 1$ a constant $C(r)$ such that we have for any $g \in W_r^1(\omega_0)$
\[
E_m(g) \leq \frac{C(r)}{m^{\frac{r}{2}}} \|g^{(r)}\varphi^r\omega_0\|_1\,, 
\]
where $C(r)$ neither depends on $m$ nor on $g$. Our function 
\[
g(s)=e^s\mathcal{F}f_{(0)}(\tfrac{s}{\gamma})=\sqrt{n}e^{(1-\frac{\sqrt{n}}{\gamma})s}, 
\]
is even smoother and we obtain
\begin{align*}
\| g^{(r)}(s)(\sqrt{s})^r e^{-s} \|_1 
&= 
\sqrt{n}\int_0^\infty \left| e^{-s} \left( 1-\frac{\sqrt{n}}{\gamma} \right)^r s^{ \frac{r}{2}} e^{\left( 1-\frac{\sqrt{n}}{\gamma} \right)s} \right| \,ds \\[1ex]
&=
\sqrt{n} \left|  1-\frac{\sqrt{n}}{\gamma}\right|^r
\int_0^\infty s^{\frac{r}{2}}e^{-\tfrac{\sqrt{n}}{\gamma}s}\,ds \\[1ex]
&=
\sqrt{n} \left|  1-\frac{\sqrt{n}}{\gamma}\right|^r \cdot \frac{\gamma}{\sqrt{n}}\cdot \left(\frac{\gamma}{\sqrt{n}}\right)^{\frac{r}{2}}
\int_0^\infty s^{\frac{r}{2}}e^{-s}\,ds \\[1ex]
&= 
\gamma \left(\frac{\sqrt{n}}{\gamma}\right)^{\frac{r}{2}}
\left|  \frac{\gamma}{\sqrt{n}}-1\right|^r \Gamma\left( \tfrac{r}{2}+1\right).
\end{align*}
Hence, we have
\[
E_m(g) \leq C(\gamma,r) \left( \frac{\sqrt{n}}{m}\right)^{\frac{r}{2}}  \left|  \frac{\gamma}{\sqrt{n}}-1\right|^r \Gamma\left( \tfrac{r}{2}+1\right)
\]
and therefore with the choice $r=2p$ and $m=n$
\[
E_n(g) \leq C(\gamma,2p) \frac{p!}{n^{\frac{p}{2}}}  \left| \frac{\gamma}{\sqrt{n}}-1\right|^{2p} \leq \frac{C(\gamma,p)}{n^{\frac{p}{2}}}\,.
\]
Choosing the coefficients $a_k$ in accordance with the coefficients $w_k^n$ that belong to a best approximation, the first statement is proved. Since
\[
\left\| \sum_{k=1}^n w_k^n \frac{1}{(\gamma-A)^k} \right\|
\leq
\left\| \frac{\sqrt{n}}{\sqrt{n}-A}\right\| + 
\left\| 
\frac{\sqrt{n}}{\sqrt{n}-A}-\sum_{k=1}^n w_k^n \frac{1}{(\gamma-A)^k}
\right\| \leq N+C(\gamma,p) = C(\gamma,p,N)
\]
the second statement is an immediate consequence.
\end{proof}

The choice $\gamma=\sqrt{n}$ in Lemma~\ref{smoothapprox} obviously gives error zero, since the resolvent $\sqrt{n}(\sqrt{n}-A)^{-1}$ can be represented exactly in the resolvent subspace $\mathcal{R}_n(\sqrt{n},A)$.

We now continue with two further lemmas that state how well a vector $v \in \mathcal{D}(A^q)$ can be approximated in resolvent Krylov subspaces of type \eqref{smoothiv}. 

\begin{lemma} \label{smoothprep} There exist operators $\tilde{S}_{n,q}$ with 
\[
\tilde{S}_{n,q}v \in \mathcal{K}_{(2q-1)(n-1)+q}((\gamma-A)^{-1}, (\gamma-A)^{-q}v)
\]
such that for all $v \in \mathcal{D}(A^q)$ and $n \in \N$ the bounds
\[
\|\tilde{S}_{n,q} v - v\| 
\leq \frac{C(\gamma,q,N)}{n^{\frac{q}{2}}}( \|v\|+\|A^qv\|)\,,
\quad \|\tilde{S}_{n,q}\| \leq C(\gamma,q,N)\,,
\]
hold true with constants $C(\gamma,q,N)$ not depending on $n$. Moreover, we have
\[
A^l\tilde{S}_{n,q} v=\tilde{S}_{n,q}A^lv 
\quad \mbox{for} \quad v \in \mathcal{D}(A^l)\,.
\]
\end{lemma}

\begin{proof}
We choose the coefficients $h_k^q$ as in Lemma~\ref{smoothop} and set
\[
\tilde{S}_{n,q} = \sum_{k=q}^{2q-1} h_k^q W_n^k\,,
\quad h_k^q = \binom{2q-1}{k} \binom{k-1}{k-q} (-1)^{k-q}\,,
\]
where the $W_n^k$ are the $k$-th powers of the operators $W_n$ in Lemma~\ref{smoothapprox}. Since these operators $W_n$ are uniformly bounded according to Lemma~\ref{smoothapprox}, it is clear that the operators $\tilde{S}_{n,q}$ are uniformly bounded with respect to $n$ as well. From the definition of $W_n$, we conclude
\[
A^l W_n v = W_n A^l v
\quad \mbox{for all} \quad v \in \mathcal{D}(A^l)\,,
\]
and hence the same holds true for $\tilde{S}_{n,q}$. For fixed $q$, let now $v \in \mathcal{D}(A^q)$ be arbitrary. It follows
\[
\|\tilde{S}_{n,q}v-v\| 
\leq \|\tilde{S}_{n,q}-H_{n,q}\| \|v\|+ \|H_{n,q}v-v\| 
\leq \|\tilde{S}_{n,q}-H_{n,q}\| \|v\| + \frac{C(q,N)}{n^{\frac{q}{2}}} \|A^qv\|
\]
with Lemma~\ref{smoothop}. Since we can write
\[
W_n^k - \left( \frac{\sqrt{n}}{\sqrt{n}-A} \right)^k=\sum_{l=0}^{k-1} W_n^{k-l-1}\left( W_n- \frac{\sqrt{n}}{\sqrt{n}-A} \right)
\left( \frac{\sqrt{n}}{\sqrt{n}-A}\right)^l,
\]
and since $W_n$ and $\sqrt{n}(\sqrt{n}-A)^{-1}$ are bounded independently of $n$, we obtain with Lemma~\ref{smoothapprox} 
\[
\biggl\|W_n^k-\left( \frac{\sqrt{n}}{\sqrt{n}-A}\right)^k\biggr\| 
\leq \frac{C(\gamma,q,N)}{n^{\frac{q}{2}}} \quad \mbox{for}\quad k=q,\ldots,2q-1\,,
\]
and therefore
\[
\|\tilde{S}_{n,q}-H_{n,q}\| \leq \frac{C(\gamma,q,N)}{n^{\frac{q}{2}}}\,.
\]
Altogether, the first bound in our lemma is proved. Due to the special form of the operators $W_n$, the statement
$\tilde{S}_{n,q}v \in \mathcal{K}_{(2q-1)n-(q-1)}((\gamma-A)^{-1}, (\gamma-A)^{-q}v)$
is clear.
\end{proof}

\begin{lemma} \label{smoothfinal} There exist operators $S_{n,q}$ with 
\[
S_{n,q}v \in \mathcal{K}_{\lfloor\frac{n}{2}\rfloor}((\gamma-A)^{-1}, (\gamma-A)^{-q}v)
\]
such that for all $v \in \mathcal{D}(A^q)$ and $n \geq 2q$, we have
\[
\|S_{n,q} v - v\| \leq \frac{C(\gamma,q,N)}{n^{\frac{q}{2}}}\left( \|v\|+\|A^qv\|\right)\,,
\quad \|S_{n,q}\| \leq C(\gamma,q,N)
\]
with $n$-independent constants and
\[
A^l S_{n,q} v = S_{n,q} A^l v 
\quad \mbox{for} \quad v \in \mathcal{D}(A^l).
\]
\end{lemma}

\begin{proof}
This is an immediate result of Lemma~\ref{smoothprep} by setting
$S_{n,q}:= \tilde{S}_{\lfloor\frac{n-2q}{2(2q-1)}\rfloor+1,q}$.
\end{proof}

Lemma~\ref{smoothfinal} basically says that one can approximate a vector $v \in \mathcal{D}(A^q)$ in about half of the Krylov subspace $\mathcal{K}_n((\gamma-A)^{-1},(\gamma-A)^{-q}v)$ and retaining a convergence according to the smoothness of $v$. The operator $S_{n,q}$ might therefore be seen as a smoothing projecting operator because of $S_{n,q}v \in \mathcal{D}(A^{2q})$ for $v\in \mathcal{D}(A^q)$.

\section{Approximation properties of the semigroup in Banach spaces} \label{sec:approxBanach}

We study in this section the best approximation of $e^Av$ in the resolvent Krylov subspace $\mathcal{K}_n((\gamma-A)^{-1},v)$ dependent on the smoothness of the vector $v$. 
\begin{theorem} \label{theo_bestapprox_expA}
Let $A$ be the generator of a bounded semigroup with bound $N$. Then we have 
\[
\inf_{r \in \frac{\mathcal{P}_{n-1}}{(\gamma-\cdot)^{n-1}}} \|e^{A}v-r(A)v \| 
\leq \frac{C(\gamma,q,N)}{n^{\frac{q}{2}}} (\|v\|+\|A^qv\|)
\quad \mathrm{for~all} \quad v \in \mathcal{D}(A^q)\,,
\]
where $C(\gamma,q,N)$ does not depend on $n$.
\end{theorem}

\begin{proof} With the smoothing operator $\tilde{v}_n:=S_{n,q}v$, $n \geq 2q$, from Lemma \ref{smoothfinal}, it follows
\begin{align*}
\| e^{A}v - r(A)v \| 
&\leq \|e^{A} (S_{n,q}v-v)\| + \|e^{A}\tilde{v}_n-r(A)v\| \\[1ex]
&\leq \frac{C(\gamma,q,N)}{n^\frac{q}{2}} (\|v\|+\|A^qv\|) + \|e^{A}\tilde{v}_n-r(A)v\|
\end{align*}
and thus
\[
\inf_{r \in \frac{\mathcal{P}_{n-1}}{(\gamma-\cdot)^{n-1}}} \|e^{A}v-r(A)v\| 
\leq \frac{C(\gamma,q,N)}{n^\frac{q}{2}} (\|v\| + \|A^qv\|) 
+ \inf_{r \in \frac{\mathcal{P}_{n-1}}{(\gamma-\cdot)^{n-1}}} \|e^{A}\tilde{v}_n-r(A)v\|\,.
\]
It remains to bound the second term on the right. Since $\tilde{v}_n \in \mathcal{K}_{\lfloor\frac{n}{2}\rfloor}((\gamma-A)^{-1}, (\gamma-A)^{-q}v)$,
we have for arbitrary coefficients $b_0,\ldots,b_{q-1}$, $a_1,\ldots,a_{\lfloor\frac{n}{2}\rfloor-q} \in \C$ that
\[
r(A)v = b_0 \tilde{v}_n + b_1 A \tilde{v}_n + \cdots + b_{q-1} A^{q-1} \tilde{v}_n 
+ \sum_{k=1}^{\lfloor\frac{n}{2}\rfloor-q} a_k \frac{1}{(\gamma-A)^k} A^q \tilde{v}_n
\in \mathcal{K}_n((\gamma-A)^{-1},v)\,.
\]
In the following, we use an approximation result given in \cite{ratkryphi11} for the so-called $\varphi$-functions
\begin{equation} \label{eq_phi}
\varphi_k(z) = \frac{1}{z^k} \biggl( e^z - \sum_{j=0}^{k-1} \frac{z^j}{j!} \biggr)\,,
\quad k = 1, 2, \ldots\,,
\end{equation}
that belong to $\widetilde{\mathcal{M}}$ and fit in the functional calculus introduced above. We leave $a_1,\ldots,a_{\lfloor\frac{n}{2}\rfloor-q}$ arbitrary and choose $b_j = 1/j!$, $j = 0, \ldots, q-1$, according to the sum in \eqref{eq_phi} to obtain
\begin{align*}
e^A \tilde{v}_n 
&- 
( b_0 + b_1 A \tilde{v}_n + \cdots + b_{q-1} A^{q-1} \tilde{v}_n )
- \sum_{k=1}^{\lfloor\frac{n}{2}\rfloor-q} a_k \frac{1}{(\gamma-A)^k} A^q \tilde{v}_n \\[-1ex]
&= 
\varphi_q(A) A^q \tilde{v}_n 
- \sum_{k=1}^{\lfloor\frac{n}{2}\rfloor-q} a_k \frac{1}{(\gamma-A)^k} A^q \tilde{v}_n\,.
\end{align*}
Now we have for $n > 2q+1$ the estimate
\begin{align*}
\biggl\| 
\varphi_q(A) A^q \tilde{v}_n
- \sum_{k=1}^{\lfloor\frac{n}{2}\rfloor-q} a_k \frac{1}{(\gamma-A)^k} A^q \tilde{v}_n \biggr\|
&\leq 
\biggl\| 
\varphi_q(A) - \sum_{k=1}^{\lfloor\frac{n}{2}\rfloor-q} a_k \frac{1}{(\gamma-A)^k}
\biggr\|
\| A^q \tilde{v}_n \| \\[1ex]
&\leq 
\frac{C(\gamma,q)}{\left(\lfloor\frac{n}{2}\rfloor-q\right)^{\frac{q}{2}}}
 \| A^q \tilde{v}_n \|
\leq 
\frac{C(\gamma,q,N)}{n^{\frac{q}{2}}} \| A^q \tilde{v}_n \|\,,
\end{align*}
where we used Theorem 4.2 in \cite{ratkryphi11} for the second inequality. Due to
\[
A^q \tilde{v}_n = A^q S_{n,q} v = S_{n,q} A^q v 
\quad \text{and} \quad
\|A^q\tilde{v}_n\| = \|S_{n,q}A^qv\| \leq C(\gamma,q,N) \|A^qv\|\,,
\]
we find for $ n > 2q+1$ that
\[
\inf_{r \in \frac{\mathcal{P}_{n-1}}{(\gamma-\cdot)^{n-1}}} 
\|e^{A} \tilde{v}_n - r(A)v\| 
\leq \frac{C(\gamma,q,N)}{n^{\frac{q}{2}}} \|A^qv\| 
\leq \frac{C(\gamma,q,N)}{n^{\frac{q}{2}}} (\|v\|+\|A^qv\|)\,. 
\]
Since this bound holds true up to finitely many numbers, adjusting the constant renders the bound true for all $n \in \N$ and our theorem is proved. 
\end{proof}

\begin{remark}
A discussion on a suitable choice of the shift $\gamma$ in the resolvents of the rational approximations to the $\varphi$-functions can be found in \cite{GG13}. 
\end{remark}

\section{Smoothness-detection by the resolvent Krylov subspace method} \label{sec:detect}

Let $A$ be a linear operator on a Hilbert space $H$ with $\mathrm{Range}(\lambda-A)=H$ for some $\lambda$ with $\mathrm{Re}\,\lambda >0$ and
\[
\mbox{Re}\,(Ax,x) \leq 0\
\quad \text{for all} \quad x \in \mathcal{D}(A)\,.
\]
By the Lumer-Phillips theorem for Hilbert spaces (e.g. Corollary 4.3.11 on page 186 in \cite{Mikl98}), $A$ is the generator of a contraction $C_0$-semigroup with
\[
\|e^{tA}\| \leq 1
\quad \text{for all} \quad t \geq 0\,,
\]
i.e. a bounded semigroup with $N=1$.
We designate by $P_n$ the orthogonal projection onto $\mathcal{K}_n((\gamma-A)^{-1},v)$ and by $\An=P_nAP_n$ the restriction of $A$ to this subspace. For simplicity, we assume in this section that $v \in \mathcal{D}(A)$, and therefore $\mathcal{K}_n((\gamma-A)^{-1},v) \subseteq \mathcal{D}(A)$. Then $\An$ also satisfies
\[
\mbox{Re}\,(\An x,x) \leq 0 
\quad \text{for all} \quad x \in \mathcal{K}_n((\gamma-A)^{-1},v)
\]
as well as $\mathrm{Range}(\lambda-\An)=H$ for some $\lambda$ with $\mathrm{Re}\,\lambda >0$, and therefore generates a contraction semigroup. Hence, our functional calculus can be applied to functions of $A$ as well as $\An$. 

The next theorem gives an upper bound for the approximation of the operator $\varphi$-function of $A$ times $v \in \mathcal{D}(A^q)$ in the resolvent Krylov subspace $\mathcal{K}_n((\gamma-A)^{-1},v)$. More precisely, we approximate $\varphi_j(A)v$ by $\varphi_j(\An)v$, where $\varphi_j$ is for $j = 1, 2, \ldots$ the $\varphi$-function defined above in \eqref{eq_phi} and $\varphi_0$ denotes the exponential function.

\begin{theorem} \label{theo_detect} Let $\An=P_nAP_n$, where $P_n$ is the orthogonal projection onto the resolvent Krylov subspace 
$\mathcal{K}_n((\gamma-A)^{-1},v)$. For $v \in \mathcal{D}(A^q)$ and $j \geq 0$, we have
\[
\|\varphi_j(A)v-\varphi_j(\An)v\| 
\leq \frac{C(\gamma,q,j)}{n^{\frac{j+q}{2}}} (\|v\|+\|A^qv\|)\,,
\]
where $C(\gamma,q,j)$ does not depend on $n$.
\end{theorem}

\begin{proof} Let $v \in \mathcal{D}(A^q)$ and $k \geq 0$ be arbitrary. To bound $\|\varphi_j(A)v-\varphi_j(\An)v\|$, we use as a first step our smoothing operator $\tilde{v}_n:=S_{n,q}v$, $n > 2q+1$ w.l.o.g., and turn this difference into
\begin{equation} \label{split_difference}
\| \varphi_j(A)v - \varphi_j(\An)v \| 
\leq \| \varphi_j(A)\tilde{v}_n - \varphi_j(\An) \tilde{v}_n \|
+ \| \big[ \varphi_j(A) - \varphi_j(\An)\big] (\tilde{v}_n-v) \|.
\end{equation}
In the following, both terms will be bounded separately. We start with the second term. Since
\[
\tilde{v}_n-v \in \mathcal{K}_{\lfloor\frac{n}{2}\rfloor+q}((\gamma-A)^{-1},v)\,,
\]
and since the resolvent Krylov subspace approximation is exact for every rational function $r(z) = p_{n-1}(z)/(\gamma-z)^{n-1}$, $p_{n-1}(z) \in \mathcal{P}_{n-1}$ (see \cite{Beckermann_Reichel09}), we have for every function $r_1(z)$ of the form
\[
r_1(z) = \sum_{k=1}^{\lfloor\frac{n}{2}\rfloor-q} a_k \frac{z^{k-1}}{(\gamma-z)^k}
\]
that
\[
r_1(A) (\tilde{v}_n-v)
= \sum_{k=1}^{\lfloor\frac{n}{2}\rfloor-q} 
a_k \frac{A^{k-1}}{(\gamma-A)^k} (\tilde{v}_n-v)
= \sum_{k=1}^{\lfloor\frac{n}{2}\rfloor-q} 
a_k \frac{\An^{k-1}}{(\gamma-\An)^k} (\tilde{v}_n-v)
= r_1(\An) (\tilde{v}_n-v)\,.
\]
Choosing $r_1$ as the best approximation to $\varphi_j$ in the sense of our functional calculus, we can estimate
\begin{align*}
\left\| \big[ \varphi_j(A) - \varphi_j(\An) \big] (\tilde{v}_n-v) \right\| 
& \leq 
\big( \|\varphi_j(A) - r_1(A)\| + \|\varphi_j(\An) - r_1(\An)\| \big) \| \tilde{v}_n-v \|  \\
& \leq \frac{C(\gamma,j)}{\left(\lfloor\frac{n}{2}\rfloor-q\right)^{\frac{j}{2}}}
\| \tilde{v}_n-v \| 
\leq 
\frac{C(\gamma,q,j)}{n^{\frac{j}{2}}} \| \tilde{v}_n - v \|\,,
\end{align*}
where we used Theorem~4.2 in \cite{ratkryphi11} for the second inequality.
Note that for the case $j = 0$ the terms $\|e^A-r_1(A)\|$ and $\|e^{\An}-r_1(\An)\|$ are just bounded by generic constants $C(\gamma,q)$. With Lemma \ref{smoothfinal}, we finally obtain
\begin{equation} \label{maintheo:eq1}
\left\| \big[ \varphi_j(A) - \varphi_j(\An) \big] (\tilde{v}_n-v) \right\| 
\leq \frac{C(\gamma,q,j)}{n^{\frac{j}{2}}} \| S_{n,q}v-v \|
\leq \frac{C(\gamma,q,j)}{n^{\frac{j+q}{2}}} (\|v\|+\|A^qv\|)\,.
\end{equation}
It remains to bound the first term on the right-hand side of \eqref{split_difference}. For this purpose, we first remark that we have
\[
b_0 \tilde{v}_n + b_1 A \tilde{v}_n + \cdots + b_{q-1} A^{q-1} \tilde{v}_n
+ \sum_{k=1}^{\lfloor\frac{n}{2}\rfloor-q} 
a_k \frac{1}{(\gamma-A)^k} A^q \tilde{v}_n \in \mathcal{K}_n((\gamma-A)^{-1},v)
\]
due to the special choice of our smoothing operator $\tilde{v}_n$. We define a second function $r_2(z)$ by
\[
r_2(z) 
= b_0 + b_1 z + \cdots + b_{q-1} z^{q-1} + r_1(z)
= b_0 + b_1 z + \cdots + b_{q-1} z^{q-1} +
\sum_{k=1}^{\lfloor\frac{n}{2}\rfloor-q} a_k \frac{z^{k-1}}{(\gamma-z)^k}\,.
\]
Then the exactness property of the resolvent Krylov subspace approximation yields
for an arbitrary choice of the coefficients $b_k$ and $a_k$ that
\[
r_2(A) \tilde{v}_n = r_2(\An) \tilde{v}_n
\]
and therefore 
\[
\| \varphi_j(A) \tilde{v}_n - \varphi_j(\An) \tilde{v}_n \| 
\leq \| \varphi_j(A) \tilde{v}_n - r_2(A)\tilde{v}_n\|
+ \| \varphi_j(\An) \tilde{v}_n - r_2(\An) \tilde{v}_n \|\,.
\]
We now choose $b_0,\ldots,b_{q-1}$ according to the Taylor series of $\varphi_j$ and obtain with 
\[
\varphi_j(z) - \sum_{k=0}^{q-1} \frac{z^k}{(j+k)!} = z^q \varphi_{j+q}(z)
\]
and Lemma~\ref{btlemma4} that 
\begin{align*}
\varphi_j(A) \tilde{v}_n - r_2(A) \tilde{v}_n 
&= 
\varphi_{j+q}(A) A^q \tilde{v}_n 
- \sum_{k=1}^{\lfloor\frac{n}{2}\rfloor-q} 
a_k \frac{A^{k-1}}{(\gamma-A)^k} A^q \tilde{v}_n\,, \\
\varphi_j(\An) \tilde{v}_n - r_2(\An) \tilde{v}_n 
&= 
\varphi_{j+q}(\An) \An^q \tilde{v}_n
- \sum_{k=1}^{\lfloor\frac{n}{2}\rfloor-q} 
a_k \frac{\An^{k-1}}{(\gamma-\An)^k} \An^q \tilde{v}_n\,.
\end{align*}
Hence, we have with our functional calculus
\begin{align*}
\| \varphi_j(A) \tilde{v}_n - r_2(A) \tilde{v}_n \| 
& \leq 
\biggl\| \varphi_{j+q}(A) - \sum_{k=1}^{\lfloor\frac{n}{2}\rfloor-q} 
a_k \frac{A^{k-1}}{(\gamma-A)^k} \biggr\| 
\| A^q \tilde{v}_n \| \\
& \leq 
\| \mathcal{F}\varphi_{j+q,(0)}(s) - 
\sum_{k=1}^{\lfloor\frac{n}{2}\rfloor-q} 
a_k(-1)^{k-1} e^{-\gamma s} L_{k-1}(\gamma s) \|_1 
\| A^q \tilde{v}_n \|\,.
\end{align*}
Exactly the same calculation holds true for $\An$ and we conclude altogether
\begin{align*}
& \| \varphi_j(A) \tilde{v}_n - \varphi_j(\An) \tilde{v}_n \| \\ 
& \leq \| \mathcal{F}\varphi_{j+q,(0)}(s) 
- \sum_{k=1}^{\lfloor\frac{n}{2}\rfloor-q} 
a_k(-1)^{k-1} e^{-\gamma s} L_{k-1}(\gamma s) \|_1 
\big( \| A^q \tilde{v}_n \| + \| \An^q \tilde{v}_n \| \big)
\end{align*}
for all choices of the coefficients $a_k$. By the exactness property, we have $\An^q\tilde{v}_n = A^q\tilde{v}_n$ and hence 
\[
\| \varphi_j(A) \tilde{v}_n - \varphi_j(\An) \tilde{v}_n \| 
\leq 2 \| \mathcal{F}\varphi_{j+q,(0)}(s) 
- \sum_{k=1}^{\lfloor\frac{n}{2}\rfloor-q} 
a_k(-1)^{k-1} e^{-\gamma s} L_{k-1}(\gamma s) \|_1 
\| A^q \tilde{v}_n \|\,.
\]
By using Theorem 4.2. in \cite{ratkryphi11}, we thus obtain
\[
\| \varphi_j(A) \tilde{v}_n - \varphi_j(\An) \tilde{v}_n \| 
\leq
\frac{C(\gamma,j)}{\left(\lfloor\frac{n}{2}\rfloor-q\right)^{\frac{j+q}{2}}}
\| A^q \tilde{v}_n \|
\leq
\frac{C(\gamma,q,j)}{n^{\frac{j+q}{2}}} \| A^q \tilde{v}_n \|\,.
\]
Since 
\[
A^q \tilde{v}_n = A^q S_{n,q}v = S_{n,q}A^qv\,, 
\quad \| A^q \tilde{v}_n \| \leq C(\gamma,q) \| A^qv \|,
\]
we finally have
\begin{equation} \label{maintheo:eq2}
\| \varphi_j(A) \tilde{v}_n - \varphi_j(\An) \tilde{v}_n \|
\leq \frac{C(\gamma,j,q)}{n^{\frac{j+q}{2}}} \| A^q v \| 
\leq \frac{C(\gamma,j,q)}{n^{\frac{j+q}{2}}} \big( \|v\| + \|A^qv\| \big)\,.
\end{equation}
With \eqref{maintheo:eq1} and \eqref{maintheo:eq2}, our statement is proved.
\end{proof}

How $\varphi_k(\An)v$ can be represented with the help of quasi-matrices when an orthonormal basis of the approximation space $\mathcal{K}_n((\gamma-A)^{-1},v)$ is known, is described in \cite{GG13}.

In view of abstract evolution equations, we are usually interested in the approximation of $e^{\tau A}v$ or $\varphi_j(\tau A)v$ for $v \in \mathcal{D}(A^q)$ and $\tau > 0$. In this case, all the above results remain valid, we only have to replace $A$ by $\tau A$ anywhere. 

Moreover, all presented results for the semigroup $e^{\tau A}$ applied to some initial data $v~\in~\mathcal{D}(A^q)$ transfer to the discrete case, that is, to the approximation of $e^{\tau A_N}\Psi_0$, where $A_N$ is the discretization matrix of the differential operator $A$ and $\Psi_0$ is the discretized initial value $v$. Since the error bounds do not depend on $\|A\|$, we obtain a convergence rate that is independent of the spatial grid.

\section{Numerical experiments} \label{sec:NumExp}

We illustrate our theoretical findings by a finite-difference discretization in Subsection~\ref{subsec:FD} as well as by a finite-element discretization of a wave equation in Subsection~\ref{subsec:FE}. Besides these illustrations, our theory provides an explanation for the behavior observed in several applications. For example, the grid-independent convergence of the rational Krylov method for the solution of Maxwell's Equations in photonic crystal modeling in \cite{Botchev16} is explained by our analysis. 

\subsection{Finite-difference discretization of the wave equation} \label{subsec:FD}

We consider the standard wave equation on the unit square $\Omega=(0,1)^2$ with homogeneous Dirichlet boundary conditions written as a system of first-order
\begin{equation} \label{wavestd_1storder}
y'(t) = \left[ \begin{array}{c}
           u(t) \\
           u'(t)
        \end{array} \right]'
      = \left[ \begin{array}{cc}
            0 & I \\
            \Delta & 0
        \end{array} \right]
        \left[ \begin{array}{c}
           u(t) \\
           u'(t)
        \end{array} \right]
      = A y(t)\,, \quad
y(0) = y_0 = \left[ \begin{array}{c}
                u_0 \\
                u_0'
             \end{array} \right].
\end{equation}
The operator $A$ on the Hilbert space  $\mathcal{D}(\sqrt{-\Delta})\times L^2(\Omega)$ equipped with the inner product  
\begin{equation} \label{skalprodhomwave}
  (v,w)=(\sqrt{-\Delta}\,v_1,\sqrt{-\Delta}\,w_1)_{L^2(\Omega)}+(v_2,w_2)_{L^2(\Omega)}\,,
\end{equation}
where $v=[v_1,v_2]^T$ and $w=[w_1,w_2]^T$, satisfies the properties of Section~\ref{sec:detect} and therefore generates a contraction semigroup (or more exactly, a $C_0$-group). The domain of the operator $A$ reads $\mathcal{D}(A)=\mathcal{D}(\Delta) \times \mathcal{D}(\sqrt{-\Delta})$ and, with standard Sobolev spaces, $\mathcal{D}(A)=(H^2(\Omega)\cap H_0^1(\Omega) ) \times H_0^1(\Omega)$. By $\| \cdot \|$, we denote the norm induced by the inner product defined in \eqref{skalprodhomwave}. 
The $q$-dependent initial values $y_0^q$ with $u_0 = u'_0 = g_0^q / \|[g_0^q, g_0^q]^T\|$, where
\[
g_0^q : \left\{
  \begin{array}{rcl}
     (0,1)^2 & \rightarrow & \R\,, \\[1ex]
     (x,y) & \mapsto & x^{2q}(1-x)^{2q}y^{2q}(1-y)^{2q}\,,
  \end{array}
\right.
\]
satisfy $y_0^q \in \mathcal{D}(A^q)$ and  $y_0^q \not\in \mathcal{D}(A^{q+1})$.

In order to illustrate and verify our theory numerically, we discretize the operator $A$ via finite differences on the grid $(ih,jh)$, $i,j=1,\ldots,d$ with $h=\frac{1}{d+1}$, which leads to a system of the type above with the matrix
\[
A_{N}=\left[
   \begin{array}{cc}
    0 & I_{N} \\
   \Delta_{N} & 0
   \end{array} 
\right] \in \R^{2N \times 2N}, \quad 
\Delta_{N}=\frac{1}{h^2}(T_d \otimes I_d + I_d \otimes T_d)\,, \quad 
T_d=\mbox{tridiag}(1,-2,1),
\]
where $\otimes$ is the Kronecker product. The matrix $\Delta_{N} \in \R^{N \times N}$ is the standard discretization with the five-point stencil for the Laplacian. We deal with the space $\R^{2N}$ equipped with the inner product  
\begin{equation} \label{skalprodfd}
  (\Psi,\Theta)_h=h^2(\sqrt{-\Delta_{N}}\,\Psi_1,\sqrt{-\Delta_{N}}\,\Theta_1)_2+h^2(\Psi_2,\Theta_2)_2 =h^2(-\Delta_{N}\,\Psi_1,\Theta_1)_2+h^2(\Psi_2,\Theta_2)_2\,,
\end{equation}
where $\Psi=[\Psi_1,\Psi_2]^T$ and $\Theta=[\Theta_1,\Theta_2]^T$, $\Psi_i, \Theta_i \in \R^{N}$, $i=1,2$, and $(\cdot, \cdot)_2$ designates the standard Euclidean inner product in $\R^{N}$. The matrix $A_{N}$ also satisfies the assumptions in Section~\ref{sec:detect}.
We define discretizations of the initial values $y_0^q$ by
\[
\Psi^q = \left[
  \begin{array}{c}
   \Psi_{1}^q \\[1ex]
   \Psi_{2}^q
  \end{array}
\right], \quad 
\Psi_i^q = \frac{1}{\|[g_0^q, g_0^q]^T\|} \big( g_0^q(ih,jh) \big)_{i,j=1}^d \in \R^{N}
\quad \text{for} \quad i =1, 2\,.
\]
For these initial values, we have 
\[
\|\Psi^q\|_h \leq C\|y_0^q\| 
\quad \text{and} \quad
\|A_{N}^q \Psi^q \|_h \leq C\|A^qy_0^q\|\,, 
\] 
where $\|\cdot\|_h$ is the norm induced by the inner product \eqref{skalprodfd} and $C$ is a generic constant that does not depend on $N$. Therefore, the error of the resolvent Krylov subspace approximation to the matrix exponential measured in the discrete norm is bounded independently of $N$ according to our Theorem~\ref{theo_detect} as
\[
 \|e^{\tau A_{N}}\Psi^q-e^{\tau \An}\Psi^q\|_h \leq \frac{C(\gamma,q)}{n^{\frac{q}{2}}}(\|y_0^q\|+\tau^q\|A^qy_0^q\|)
\]
with $\An = P_n A_N P_n$, where $P_n$ is the orthogonal projection onto $\mathcal{K}_n( (\gamma-\tau A_N)^{-1}, \Psi^q)$. 
Note that the right-hand side with the continuous values, does not depend on $N$.
This worst case sublinear convergence can be clearly observed in the numerical results in Figure~\ref{fig_fd_limit} with $\tau=0.5$, where the error in the discrete norm is shown over the dimension of the Krylov subspace for discretizations of $A$ with $d=15,31,127,255,511,1023,2047$, leading to matrices $A_{N}$ from size $450 \times 450$ to size $8,380,418 \times 8,380,418$. For the smaller matrices with $d=15$ and $d=31$, the convergence is faster than predicted. For the remaining matrices up to size $8,380,418 \times 8,380,418$, the predicted sublinear convergence can be seen. Furthermore, the convergence is faster for the smoother initial value $\Psi^4$ on the right-hand side of Figure~\ref{fig_fd_limit} compared to $\Psi^2$ on the left-hand side, which also fits perfectly to our theorem. For a suitable space discretization, this behavior is always to be expected. In Figure~\ref{fig_fdmg_timings}, the error of the backward Euler method and the resolvent Krylov subspace method, respectively, is shown versus the computing time in minutes for the discretization with dimension $8,380,418$. The resolvents have been computed by a multigrid method and the exact solution for the computation of the error has been calculated by a discrete fast Fourier transform.  
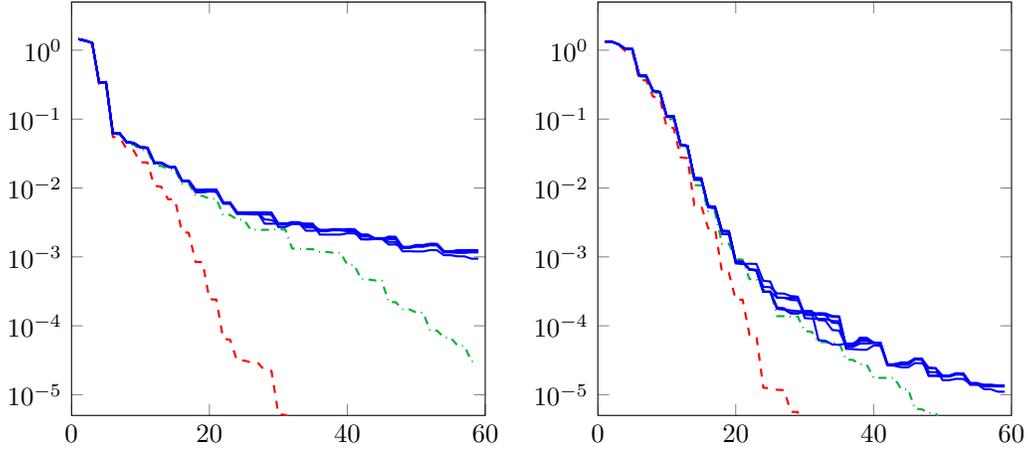
\begin{figure} \centering
\input{./wave_fd_limit_smooth2.tikz}~
\input{./wave_fd_limit_smooth4.tikz}
\caption{Error versus Krylov subspace dimension for the approximation of $e^{\tau A_{N}}\Psi^q$ with $q=2$ (left-hand side) and $q=4$ (right-hand side) for matrices of dimension 8,380,418; 2,093,058; 522,242; 130,050; 32,258; 7938 (blue solid lines); 1922 (green dash-dotted line); 450 (red dashed line).}
\label{fig_fd_limit}
\end{figure}

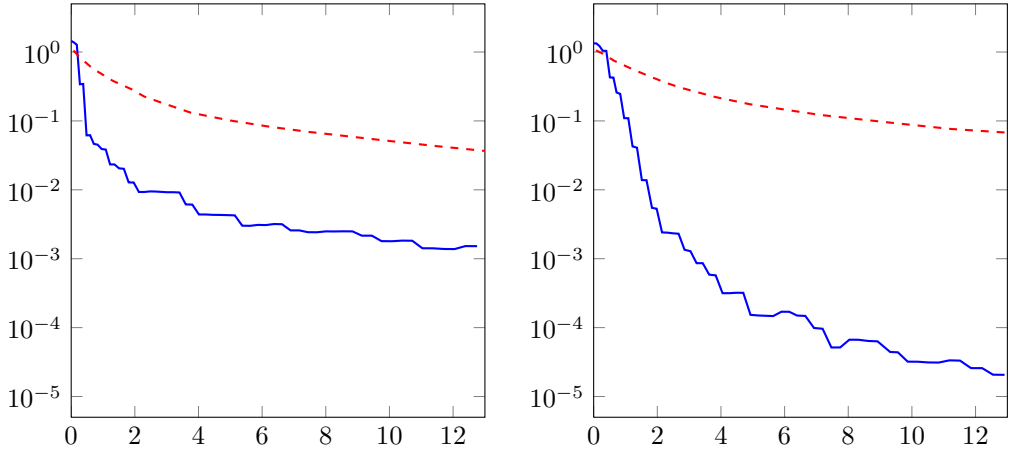
\begin{figure} \centering
\input{./wave_fdmg_timings_smooth2.tikz}~~~
\input{./wave_fdmg_timings_smooth4.tikz}
\caption{Error versus time in minutes for the approximation of $e^{\tau A_{N}}\Psi^q$ with $q=2$ (left-hand side) and $q=4$ (right-hand side) for the matrix of dimension $8,380,418$ using the implicit Euler method (red dashed line) and the resolvent Krylov subspace method (blue solid line).}
\label{fig_fdmg_timings}
\end{figure}

\subsection{Finite-element discretization of a wave equation on a non-standard domain} \label{subsec:FE}

For a trapezoidal domain $\Omega$ with a slit in the upper half, we consider the two-dimensional wave equation
\[
u''(t) = \Delta u(t) - u(t)\,, \quad u(0) = u_0\,, \quad u'(0) = u'_0
\]
with homogeneous Neumann boundary conditions which can be represented by the first-oder system 
\begin{equation} \label{wave_1storder}
y'(t) = \left[ \begin{array}{c}
           u(t) \\
           u'(t)
        \end{array} \right]'
      = \left[ \begin{array}{cc}
            0 & I \\
            \Delta - I & 0
        \end{array} \right]
        \left[ \begin{array}{c}
           u(t) \\
           u'(t)
        \end{array} \right]
      = A y(t)\,, \quad
y(0) = y_0 = \left[ \begin{array}{c}
                u_0 \\
                u_0'
             \end{array} \right],
\end{equation}
where $\Delta$ is the Laplacian including the boundary conditions. 
It can be shown that $A$ is the generator of a contraction semigroup with respect to the inner product analogous to \eqref{skalprodhomwave}, where the operator $\sqrt{-\Delta}$ is replaced by $\sqrt{I-\Delta}$. The domain of $A$ is given by $\mathcal{D}(A) = \mathcal{D}(\Delta) \times \mathcal{D}(\sqrt{I-\Delta})$. We use here the equivalent norms that are commonly used for finite elements. For $v = [ v_1, v_2 ]^T$ and $w = [ w_1, w_2 ]^T$ with $v_1, w_1 \in H^1(\Omega)$, $\nabla_n v_1 = \nabla_n w_1 = 0$ on $\partial \Omega$ and $v_2, w_2 \in L^2(\Omega)$, the inner product reads
\[
(v,w)= (\nabla v_1, \nabla w_1)_{L^2(\Omega)}+ (v_1,w_1)_{L^2(\Omega)} + (v_2,w_2)_{L^2(\Omega)}\,.
\]
We solve equation \eqref{wave_1storder} numerically by using finite elements with $N$ linear nodal basis functions $\phi_k \in H^1(\Omega)$, $k = 1, \ldots, N$. This leads to the semi-discrete formulation
\begin{equation} \label{wave_discrete}
\left[ \begin{array}{cc}
   M & 0 \\
   0 & M        
\end{array} \right]
\left[ \begin{array}{c}
   \Psi_1(t) \\
   \Psi_2(t)       
\end{array} \right]'
=
\left[ \begin{array}{cc}
   0 & M \\
   S-M & 0        
\end{array} \right]
\left[ \begin{array}{c}
   \Psi_1(t) \\
   \Psi_2(t)       
\end{array} \right], \quad
\left[ \begin{array}{c}
   \Psi_1(0) \\
   \Psi_2(0)       
\end{array} \right]
=
\left[ \begin{array}{c}
   \Psi_{1,0} \\
   \Psi_{2,0}       
\end{array} \right],
\end{equation}
where the vectors $\Psi_1(t), \Psi_2(t) \in \R^N$ are the coordinate vectors for $u(t)$ and $u'(t)$. The mass matrix $M \in \R^{N \times N}$ and the stiffness matrix $S \in \R^{N \times N}$ are defined by $(M)_{jk} = (\phi_j,\phi_k)_{L^2(\Omega)}$ and $(S)_{jk} = -(\nabla\phi_j,\nabla\phi_k)_{L^2(\Omega)}$ for $j, k = 1, \ldots, N$. Multiplying \eqref{wave_discrete} from the left with the inverse of the block diagonal matrix $\text{diag}(M,M)$, we end up with the initial value problem 
\[
\Psi'(t) = A_N \Psi(t)\,, \quad \Psi(0) = \Psi_0
\]
with $\Psi(t) = [ \Psi_1(t), \Psi_2(t) ]^T$ and $\Psi_0 = [ \Psi_{1,0}, \Psi_{2,0} ]^T$. For the initial data $y_0^q$ in \eqref{wave_1storder}, we choose $u_0 = u'_0 = g_0^q/\|[g_0^q,g_0^q]^T\|$ with
\[
g_0^q : \left\{
  \begin{array}{rcl}
     \Omega & \rightarrow & \R\,, \\[1ex]
     (x,y) & \mapsto & (x+1)^{2q} (x-1)^{2q} (y-1)^{2q} (y-2)^{2q}\,,
  \end{array}
\right.
\]
where $y_0^q \in \mathcal{D}(A^q)$ but $y_0^q \not\in \mathcal{D}(A^{q+1})$. For the initial data used in our numerical experiment, we evolved $y_0^q$ with the given wave equation for time $\tau = 0.5$. Its discrete counterpart, depicted in Figure~\ref{fig_wave_initial&tau3} on the left-hand side, we denote by $\Psi_0^q$. Moreover, we show on the right-hand side the numerical solution for time $\tau = 3$. In Figure~\ref{fig_wave_error}, the obtained error curves are plotted for a coarse grid with 31,232 triangles and 160,323 nodes as well as a fine grid with 499,712 triangles and 251,520 nodes. As parameters, we have chosen the time step size $\tau = 0.05$ and the smoothness indices $q=1$ and $q=3$.

In Figure~\ref{fig_wave_error}, the obtained error curves for the approximation of $e^{\tau A_N} \Psi_0^q$ in the Krylov subspace $\mathcal{K}_n((15-\tau A_N)^{-1}, \Psi_0^q)$ are plotted for the coarse grid as well as for the a fine grid. As parameters, we have chosen the time step size $\tau = 0.05$ and the smoothness indices $q = 1$ and $q = 3$. The linear systems with the matrix $15 - \tau A_N$ were solved again by a multigrid method.

\begin{figure} \centering
\includegraphics[width=7cm, trim=5.2cm 9.2cm 2.5cm 11.2cm, clip]{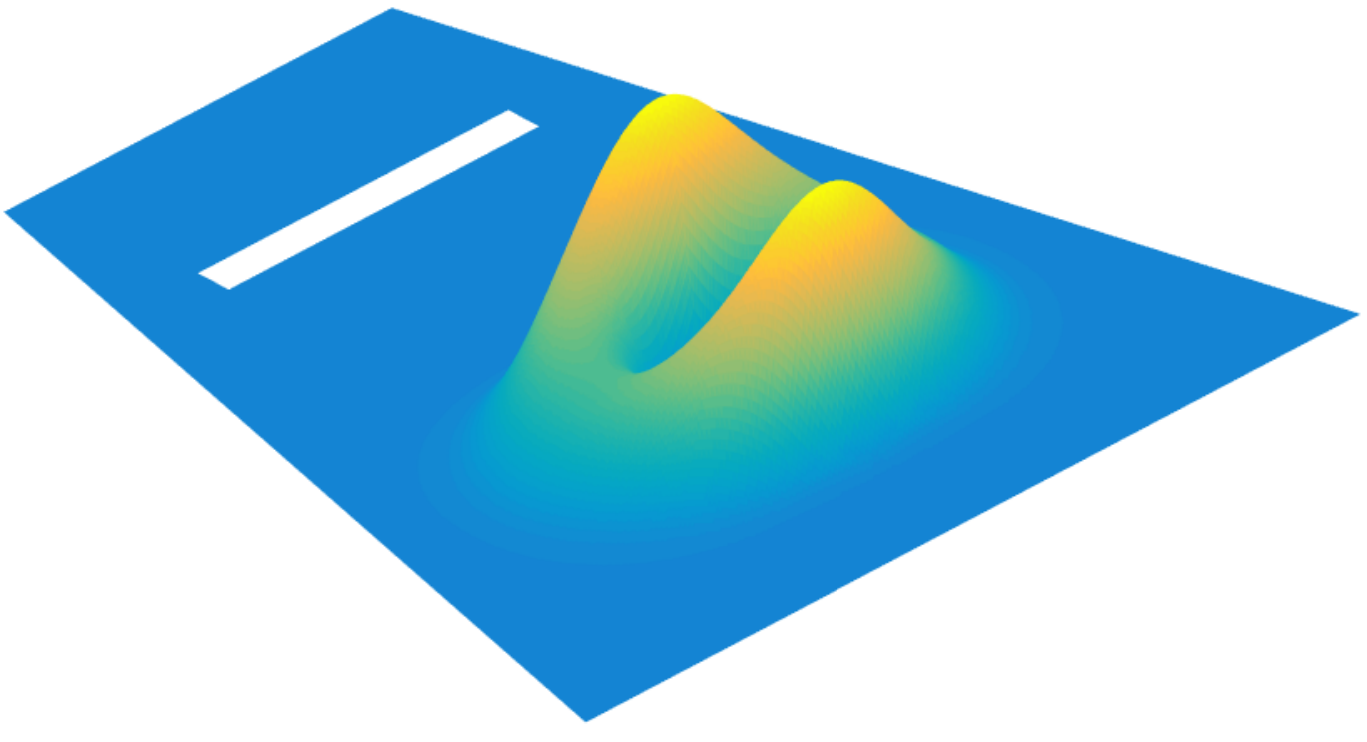} 
\hspace*{0.5cm}
\includegraphics[width=7cm, trim=5.2cm 8.5cm 2.5cm 10cm, clip]{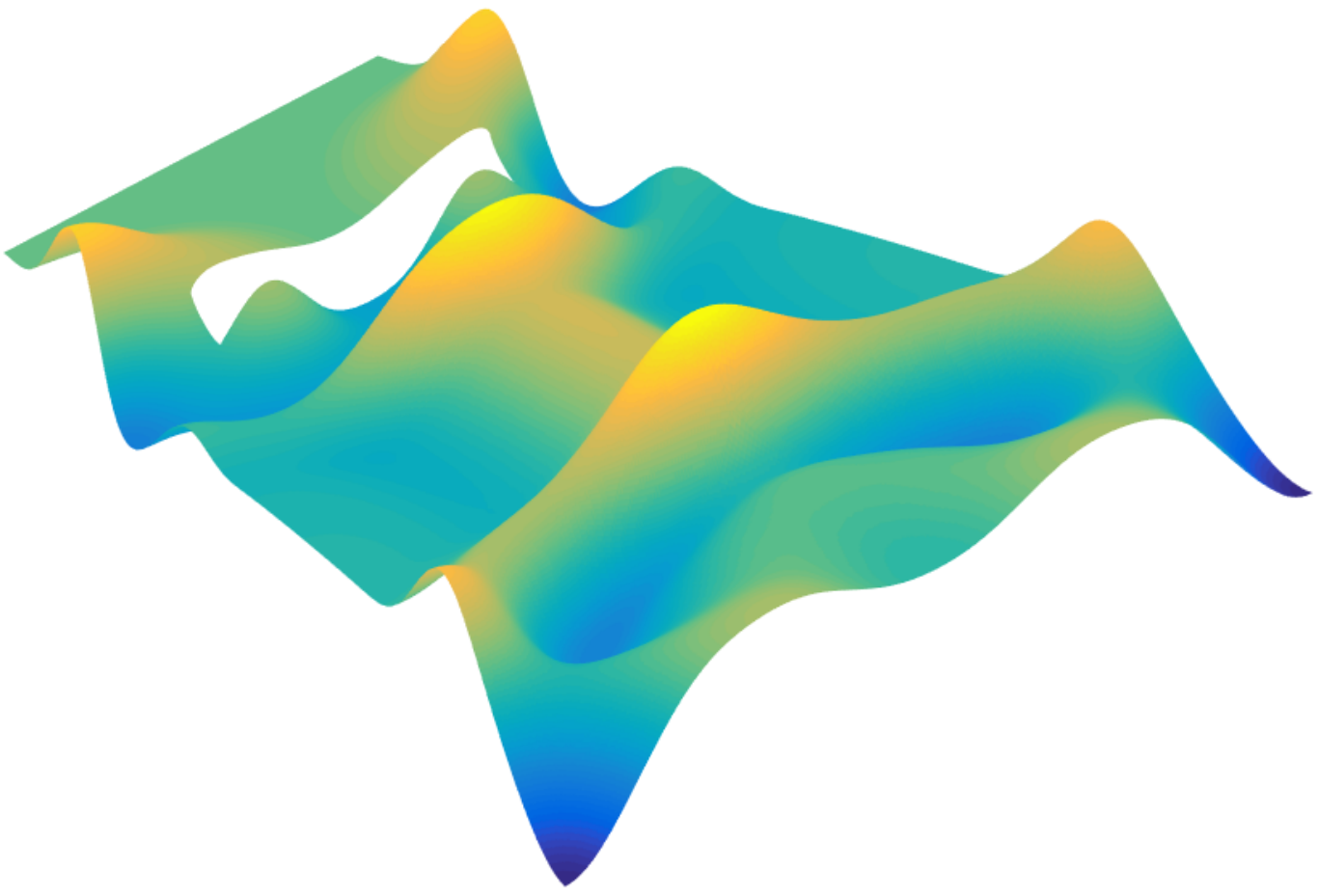}
\caption{Initial value $\Psi_{1,0}^3$ (left) and the corresponding numerical solution $\Psi_1^3$ at time $\tau=3$ (right).}
\label{fig_wave_initial&tau3}
\end{figure}

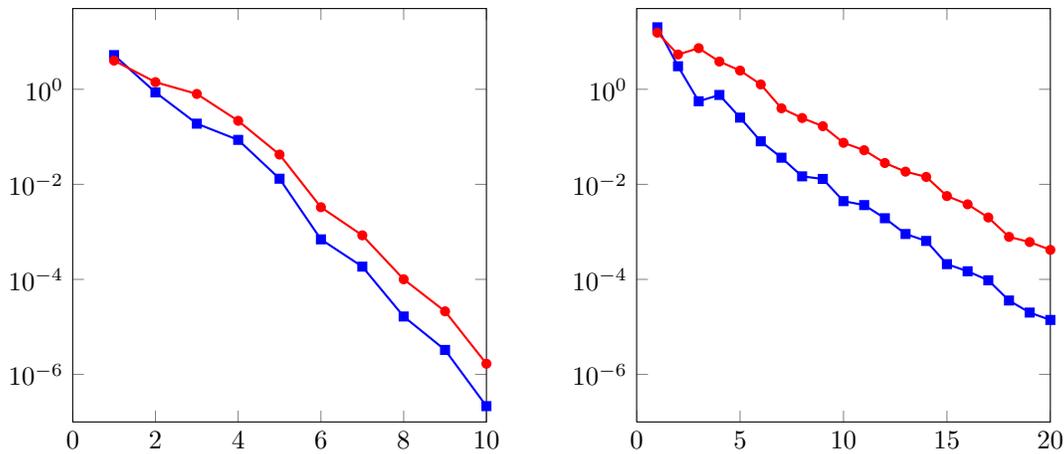
\begin{figure} \centering
\input{./Wave_error_coarse.tikz}
\hspace*{0.5cm}
\input{./Wave_error_fine.tikz}
\caption{Plot of the error versus dimension of the Krylov subspace $\mathcal{K}_m((15-A_N)^{-1}, \Psi_0^q)$ for $N = 16,032$ (left) and $N = 251,520$ (right), $\tau=0.05$, and initial vectors $\Psi_0^q$ with $q = 1, 3$ (circle-, square-marked line).}
\label{fig_wave_error}
\end{figure}

 
\section{Conclusion} \label{sec:conclusion}
In this work, we could show that the resolvent Krylov subspace method is suitable for the approximation of a large set of operator functions. For the semigroup and related operator functions, convergence rates dependent on the smoothness of the initial data have been presented. In contrast to standard methods, the faster convergence for smoother initial data is automatic, that is, the method does not need to be altered in any way to achieve the faster convergence for smoother initial data. The theoretical findings have been illustrated by numerical experiments.   

\section*{Acknowledgements}

This work has been supported by the Deutsche Forschungsgemeinschaft (DFG) via GR~3787/1-1.

\bibliographystyle{plain} 
\bibliography{lit}

\end{document}

%% file: Schroe_17_0p02_20.tikz
%
%
\begin{tikzpicture}

\begin{axis}[
width=8cm, height=6cm, scale only axis,
xmin=0, xmax=15, ymode=log, ymin=1e-11, ymax=1e-1,
ytick={1,1e-2,1e-4,1e-6,1e-8,1e-10},yminorticks=false
]
\node[blue] at (axis cs: 13.9,2e-9) {\small $q=8$};
\node[blue] at (axis cs: 13.9,3e-8) {\small $q=6$};
\node[blue] at (axis cs: 13.9,5e-7) {\small $q=4$};
\node[blue] at (axis cs: 13.9,2e-5) {\small $q=2$};
\addplot[mark=otimes*, mark size=1.5pt, color=blue, thick] 
  table[row sep=crcr]{%
1	0.0286806880062095\\
2	0.00222205407487618\\
3	0.000318198667704703\\
4	0.000154210266186596\\
5	8.52362436712081e-05\\
6	4.91652552897672e-05\\
7	3.21983799522681e-05\\
8	2.35318555685361e-05\\
9	1.79051516149416e-05\\
10	1.20869729606534e-05\\
11	7.95290583934606e-06\\
12	6.74404099214301e-06\\
13	5.26781080219631e-06\\
14	3.6547942959624e-06\\
15	3.06433920946873e-06\\
16	2.43983440937878e-06\\
17	1.91271493150894e-06\\
18	1.83828992418223e-06\\
19	1.55147548670563e-06\\
20	1.64067065023074e-06\\
};
\addplot[color=red, thick, dashed] 
  table[row sep=crcr]{%
1	0.00370625613728846\\
2	0.00196050101330424\\
3	0.00133269940129977\\
4	0.00100969656916717\\
5	0.000812868313446877\\
6	0.000680345871749976\\
7	0.00058502900111226\\
8	0.00051317142864641\\
9	0.000457057308377902\\
10	0.000412020907885705\\
11	0.000375074843965014\\
12	0.000344217587392678\\
13	0.000318057726047202\\
14	0.000295597915698593\\
15	0.000276104573967745\\
16	0.000259026078011352\\
17	0.000243939631392577\\
18	0.000230515723956171\\
19	0.000218493743496728\\
20	0.000207664859361655\\
};
\addplot[mark=otimes*, mark size=1.5pt, color=red, thick, only marks] 
  table[row sep=crcr]{%
1	0.00370625613728846\\
2	0.00196050101330424\\
3	0.00133269940129977\\
4	0.00100969656916717\\
5	0.000812868313446877\\
6	0.000680345871749976\\
7	0.00058502900111226\\
8	0.00051317142864641\\
9	0.000457057308377902\\
10	0.000412020907885705\\
11	0.000375074843965014\\
12	0.000344217587392678\\
13	0.000318057726047202\\
14	0.000295597915698593\\
15	0.000276104573967745\\
16	0.000259026078011352\\
17	0.000243939631392577\\
18	0.000230515723956171\\
19	0.000218493743496728\\
20	0.000207664859361655\\
};
\addplot[mark=triangle*, mark size=2pt, color=blue, thick] 
  table[row sep=crcr]{%
1	0.0448864886947235\\
2	0.00404675473805381\\
3	0.000490894492894326\\
4	7.21089996636432e-05\\
5	2.54649883856642e-05\\
6	1.21696053739089e-05\\
7	5.21424532738796e-06\\
8	2.43222579243809e-06\\
9	1.25785631084636e-06\\
10	7.12765516792953e-07\\
11	4.28361885598818e-07\\
12	2.57564676527893e-07\\
13	1.51493531140825e-07\\
14	9.9516590392499e-08\\
15	6.7563190290491e-08\\
16	4.4211026256153e-08\\
17	4.37097972200826e-08\\
18	2.92059380919711e-08\\
19	2.94029461864028e-08\\
20	2.12836576225443e-08\\
};
\addplot[color=red, thick, dashed] 
  table[row sep=crcr]{%
1	0.00907620546774708\\
2	0.00488708667321387\\
3	0.00332976773286676\\
4	0.00252218181929314\\
5	0.00202900642570404\\
6	0.00169683324602477\\
7	0.00145798116893754\\
8	0.00127800624227143\\
9	0.00113754491932398\\
10	0.00102488027903232\\
11	0.000932508667633654\\
12	0.000855402980125778\\
13	0.000790069231970225\\
14	0.000734003722371879\\
15	0.000685365518637938\\
16	0.000642770823402875\\
17	0.000605159410338751\\
18	0.000571705293395198\\
19	0.000541755436235846\\
20	0.000514786745692867\\
};
\addplot[mark=triangle*, mark size=2pt, color=red, thick, only marks] 
  table[row sep=crcr]{%
1	0.00907620546774708\\
2	0.00488708667321387\\
3	0.00332976773286676\\
4	0.00252218181929314\\
5	0.00202900642570404\\
6	0.00169683324602477\\
7	0.00145798116893754\\
8	0.00127800624227143\\
9	0.00113754491932398\\
10	0.00102488027903232\\
11	0.000932508667633654\\
12	0.000855402980125778\\
13	0.000790069231970225\\
14	0.000734003722371879\\
15	0.000685365518637938\\
16	0.000642770823402875\\
17	0.000605159410338751\\
18	0.000571705293395198\\
19	0.000541755436235846\\
20	0.000514786745692867\\
};
\addplot[mark=square*, mark size=1.5pt, color=blue, thick] 
  table[row sep=crcr]{%
1	0.059485195570683\\
2	0.00763847323930575\\
3	0.0013612767443652\\
4	0.00055045238099658\\
5	0.000109232909920436\\
6	1.48512768430848e-05\\
7	1.49126249509693e-06\\
8	2.76174755589866e-07\\
9	9.98577211533384e-08\\
10	4.87084990319199e-08\\
11	3.33427342539725e-08\\
12	2.23961489256355e-08\\
13	1.23860097287092e-08\\
14	6.78151267493523e-09\\
15	3.85699050838504e-09\\
16	3.75823445480776e-09\\
17	2.31088090443969e-09\\
18	2.32825781969419e-09\\
19	1.53720782979949e-09\\
20	1.56320542940698e-09\\
};
\addplot[color=red, thick, dashed] 
  table[row sep=crcr]{%
1	0.0161514271920211\\
2	0.00918567410082807\\
3	0.00639303283219622\\
4	0.00489550948110899\\
5	0.003963935842189\\
6	0.00332919992946939\\
7	0.00286920406890648\\
8	0.00252064430315456\\
9	0.00224746104300097\\
10	0.00202762045531642\\
11	0.00184690549555341\\
12	0.00169573467763548\\
13	0.00156741651284698\\
14	0.00145713700526666\\
15	0.0013613447169183\\
16	0.00127736254538186\\
17	0.00120313438597078\\
18	0.00113705501748101\\
19	0.00107785298257364\\
20	0.00102450815933803\\
};
\addplot[mark=square*, mark size=1.5pt, color=red, thick, only marks] 
  table[row sep=crcr]{%
1	0.0161514271920211\\
2	0.00918567410082807\\
3	0.00639303283219622\\
4	0.00489550948110899\\
5	0.003963935842189\\
6	0.00332919992946939\\
7	0.00286920406890648\\
8	0.00252064430315456\\
9	0.00224746104300097\\
10	0.00202762045531642\\
11	0.00184690549555341\\
12	0.00169573467763548\\
13	0.00156741651284698\\
14	0.00145713700526666\\
15	0.0013613447169183\\
16	0.00127736254538186\\
17	0.00120313438597078\\
18	0.00113705501748101\\
19	0.00107785298257364\\
20	0.00102450815933803\\
};
\addplot[mark=x, mark size=2.5pt, color=blue, thick] 
  table[row sep=crcr]{%
1	0.0726575103599419\\
2	0.011695361456652\\
3	0.00294843348840564\\
4	0.00123758187885074\\
5	0.000628853801835491\\
6	0.00010489900991902\\
7	3.37139394687439e-06\\
8	9.73573921945625e-07\\
9	2.80819377316798e-07\\
10	6.96251812525187e-08\\
11	1.59507687488257e-08\\
12	3.41642392750379e-09\\
13	8.61990313042084e-10\\
14	3.93858531674125e-10\\
15	2.96023034940419e-10\\
16	9.00356308254539e-11\\
17	8.82099552753565e-11\\
18	1.2709396271734e-10\\
19	3.62010119099493e-11\\
20	3.61375234238765e-11\\
};
\addplot[color=red, thick, dashed] 
  table[row sep=crcr]{%
1	0.0235314984994111\\
2	0.014031543087745\\
3	0.00998548903269797\\
4	0.00774338548753359\\
5	0.00632019135516042\\
6	0.00533731492753852\\
7	0.00461816238030739\\
8	0.00406933292580152\\
9	0.00363682317994016\\
10	0.0032872518092453\\
11	0.00299888159274845\\
12	0.0027569535035247\\
13	0.00255109664649522\\
14	0.00237381142340278\\
15	0.00221954084478892\\
16	0.00208408019579077\\
17	0.00196418953104204\\
18	0.00185733213278738\\
19	0.00176149363633629\\
20	0.00167505422632234\\
};
\addplot[mark=x, mark size=2.5pt, color=red, thick, only marks] 
  table[row sep=crcr]{%
1	0.0235314984994111\\
2	0.014031543087745\\
3	0.00998548903269797\\
4	0.00774338548753359\\
5	0.00632019135516042\\
6	0.00533731492753852\\
7	0.00461816238030739\\
8	0.00406933292580152\\
9	0.00363682317994016\\
10	0.0032872518092453\\
11	0.00299888159274845\\
12	0.0027569535035247\\
13	0.00255109664649522\\
14	0.00237381142340278\\
15	0.00221954084478892\\
16	0.00208408019579077\\
17	0.00196418953104204\\
18	0.00185733213278738\\
19	0.00176149363633629\\
20	0.00167505422632234\\
};
\end{axis}
\end{tikzpicture}%

%% file: wave_fd_limit_smooth2.tikz
\begin{tikzpicture}

\begin{axis}[
width=5.5cm, height=5.5cm, scale only axis,
xmin=0, xmax=60, ymode=log, ymin=5e-6, ymax=0.5e1,
ytick={1,1e-1,1e-2,1e-3,1e-4,1e-5,1e-6}, yminorticks=false
]
\addplot[color=red, thick, dashed] coordinates {
(1, 1.460525e+00)
(2, 1.341170e+00)
(3, 1.242310e+00)
(4, 3.233973e-01)
(5, 3.271920e-01)
(6, 5.492924e-02)
(7, 5.418143e-02)
(8, 3.661023e-02)
(9, 3.517172e-02)
(10, 2.357709e-02)
(11, 2.351782e-02)
(12, 1.077487e-02)
(13, 1.042573e-02)
(14, 6.861581e-03)
(15, 6.852775e-03)
(16, 2.315284e-03)
(17, 2.234010e-03)
(18, 8.436665e-04)
(19, 8.388886e-04)
(20, 2.435626e-04)
(21, 2.392884e-04)
(22, 6.480727e-05)
(23, 6.273721e-05)
(24, 3.130162e-05)
(25, 3.129325e-05)
(26, 2.952542e-05)
(27, 2.929159e-05)
(28, 2.259259e-05)
(29, 2.214768e-05)
(30, 5.224564e-06)
(31, 5.080688e-06)
(32, 4.417084e-06)
(33, 4.316277e-06)
(34, 2.066586e-06)
(35, 2.061319e-06)
(36, 1.770080e-06)
(37, 1.746133e-06)
(38, 6.870776e-07)
(39, 6.689181e-07)
(40, 2.673625e-07)
(41, 2.588341e-07)
(42, 2.096750e-07)
(43, 2.043340e-07)
(44, 9.776572e-08)
(45, 9.783717e-08)
(46, 3.391414e-08)
(47, 3.351886e-08)
(48, 2.663347e-08)
(49, 2.617866e-08)
(50, 6.166084e-09)
(51, 5.962271e-09)
(52, 4.651886e-09)
(53, 4.503428e-09)
(54, 1.423078e-09)
(55, 1.396136e-09)
(56, 1.254232e-09)
(57, 1.251753e-09)
(58, 5.672446e-10)
(59, 5.667269e-10)
};
\addplot[color=darkgreen, dashdotted, thick] coordinates {
(1, 1.451799e+00)
(2, 1.368926e+00)
(3, 1.266091e+00)
(4, 3.360985e-01)
(5, 3.400452e-01)
(6, 6.010357e-02)
(7, 6.009593e-02)
(8, 4.395845e-02)
(9, 4.261261e-02)
(10, 3.590543e-02)
(11, 3.533947e-02)
(12, 2.092506e-02)
(13, 2.093467e-02)
(14, 1.887636e-02)
(15, 1.851828e-02)
(16, 1.163068e-02)
(17, 1.159965e-02)
(18, 7.725722e-03)
(19, 7.674382e-03)
(20, 7.003468e-03)
(21, 6.958089e-03)
(22, 4.052673e-03)
(23, 4.055313e-03)
(24, 3.490478e-03)
(25, 3.454317e-03)
(26, 2.484137e-03)
(27, 2.471114e-03)
(28, 2.470970e-03)
(29, 2.472159e-03)
(30, 2.511586e-03)
(31, 2.497582e-03)
(32, 1.321533e-03)
(33, 1.311794e-03)
(34, 1.298736e-03)
(35, 1.291962e-03)
(36, 1.180958e-03)
(37, 1.179880e-03)
(38, 1.146300e-03)
(39, 1.144349e-03)
(40, 7.725814e-04)
(41, 7.640091e-04)
(42, 4.771432e-04)
(43, 4.710522e-04)
(44, 4.552321e-04)
(45, 4.504815e-04)
(46, 2.204393e-04)
(47, 2.201305e-04)
(48, 1.744414e-04)
(49, 1.728164e-04)
(50, 1.557011e-04)
(51, 1.538168e-04)
(52, 8.661910e-05)
(53, 8.508350e-05)
(54, 6.816027e-05)
(55, 6.666873e-05)
(56, 5.174874e-05)
(57, 5.077927e-05)
(58, 3.054782e-05)
(59, 3.054672e-05)
};
\addplot[color=blue, thick] coordinates {
(1, 1.449572e+00)
(2, 1.375852e+00)
(3, 1.272020e+00)
(4, 3.393177e-01)
(5, 3.432846e-01)
(6, 6.137383e-02)
(7, 6.155959e-02)
(8, 4.582309e-02)
(9, 4.454989e-02)
(10, 3.835517e-02)
(11, 3.767934e-02)
(12, 2.289348e-02)
(13, 2.288138e-02)
(14, 2.017379e-02)
(15, 1.985292e-02)
(16, 1.253154e-02)
(17, 1.248605e-02)
(18, 8.638485e-03)
(19, 8.625917e-03)
(20, 8.891273e-03)
(21, 8.773039e-03)
(22, 5.909821e-03)
(23, 5.888617e-03)
(24, 4.193328e-03)
(25, 4.189137e-03)
(26, 4.144737e-03)
(27, 4.121400e-03)
(28, 4.127344e-03)
(29, 4.100827e-03)
(30, 2.714794e-03)
(31, 2.712914e-03)
(32, 2.996190e-03)
(33, 2.982773e-03)
(34, 3.055523e-03)
(35, 3.031281e-03)
(36, 2.106604e-03)
(37, 2.106115e-03)
(38, 2.105178e-03)
(39, 2.104576e-03)
(40, 2.199430e-03)
(41, 2.190240e-03)
(42, 1.994271e-03)
(43, 1.985851e-03)
(44, 1.873628e-03)
(45, 1.868578e-03)
(46, 1.652229e-03)
(47, 1.645372e-03)
(48, 1.206788e-03)
(49, 1.206741e-03)
(50, 1.198855e-03)
(51, 1.198478e-03)
(52, 1.254636e-03)
(53, 1.251004e-03)
(54, 1.059186e-03)
(55, 1.056845e-03)
(56, 1.019776e-03)
(57, 1.017140e-03)
(58, 9.416913e-04)
(59, 9.410823e-04)
};
\addplot[blue, thick] coordinates {
(1, 1.449012e+00)
(2, 1.377583e+00)
(3, 1.273502e+00)
(4, 3.401257e-01)
(5, 3.440965e-01)
(6, 6.169302e-02)
(7, 6.192753e-02)
(8, 4.629404e-02)
(9, 4.504141e-02)
(10, 3.892837e-02)
(11, 3.822989e-02)
(12, 2.326958e-02)
(13, 2.324999e-02)
(14, 2.041984e-02)
(15, 2.010862e-02)
(16, 1.274402e-02)
(17, 1.268667e-02)
(18, 9.127405e-03)
(19, 9.124768e-03)
(20, 9.371002e-03)
(21, 9.255935e-03)
(22, 6.085890e-03)
(23, 6.064902e-03)
(24, 4.362004e-03)
(25, 4.361297e-03)
(26, 4.373195e-03)
(27, 4.347091e-03)
(28, 4.345064e-03)
(29, 4.316565e-03)
(30, 2.948738e-03)
(31, 2.939268e-03)
(32, 3.202285e-03)
(33, 3.177525e-03)
(34, 2.924398e-03)
(35, 2.908273e-03)
(36, 2.357502e-03)
(37, 2.353416e-03)
(38, 2.450161e-03)
(39, 2.442792e-03)
(40, 2.526948e-03)
(41, 2.515027e-03)
(42, 2.113731e-03)
(43, 2.113069e-03)
(44, 1.843474e-03)
(45, 1.839640e-03)
(46, 1.819483e-03)
(47, 1.811455e-03)
(48, 1.344036e-03)
(49, 1.342149e-03)
(50, 1.409032e-03)
(51, 1.404509e-03)
(52, 1.471839e-03)
(53, 1.465359e-03)
(54, 1.161345e-03)
(55, 1.160921e-03)
(56, 1.144426e-03)
(57, 1.142551e-03)
(58, 1.165053e-03)
(59, 1.161313e-03)
};
\addplot[color=blue, thick] coordinates {
(1, 1.448872e+00)
(2, 1.378015e+00)
(3, 1.273872e+00)
(4, 3.403279e-01)
(5, 3.442996e-01)
(6, 6.177297e-02)
(7, 6.201968e-02)
(8, 4.641215e-02)
(9, 4.516481e-02)
(10, 3.906953e-02)
(11, 3.836571e-02)
(12, 2.335625e-02)
(13, 2.333480e-02)
(14, 2.047715e-02)
(15, 2.016796e-02)
(16, 1.280387e-02)
(17, 1.274236e-02)
(18, 9.261838e-03)
(19, 9.261007e-03)
(20, 9.473860e-03)
(21, 9.361615e-03)
(22, 6.132067e-03)
(23, 6.109680e-03)
(24, 4.410664e-03)
(25, 4.410408e-03)
(26, 4.404672e-03)
(27, 4.376202e-03)
(28, 4.391247e-03)
(29, 4.361384e-03)
(30, 3.025167e-03)
(31, 3.013343e-03)
(32, 3.097899e-03)
(33, 3.081055e-03)
(34, 2.397671e-03)
(35, 2.394891e-03)
(36, 2.434681e-03)
(37, 2.427807e-03)
(38, 2.465696e-03)
(39, 2.458624e-03)
(40, 2.529300e-03)
(41, 2.520275e-03)
(42, 2.060719e-03)
(43, 2.060038e-03)
(44, 1.821166e-03)
(45, 1.816128e-03)
(46, 1.857219e-03)
(47, 1.848602e-03)
(48, 1.382471e-03)
(49, 1.379859e-03)
(50, 1.486097e-03)
(51, 1.480709e-03)
(52, 1.528373e-03)
(53, 1.522309e-03)
(54, 1.194856e-03)
(55, 1.193786e-03)
(56, 1.186110e-03)
(57, 1.184390e-03)
(58, 1.228009e-03)
(59, 1.225201e-03)
};
\addplot[color=blue, thick] coordinates {
(1, 1.448837e+00)
(2, 1.378123e+00)
(3, 1.273964e+00)
(4, 3.403784e-01)
(5, 3.443504e-01)
(6, 6.179296e-02)
(7, 6.204274e-02)
(8, 4.644170e-02)
(9, 4.519569e-02)
(10, 3.910469e-02)
(11, 3.839956e-02)
(12, 2.337746e-02)
(13, 2.335556e-02)
(14, 2.049122e-02)
(15, 2.018250e-02)
(16, 1.281932e-02)
(17, 1.275669e-02)
(18, 9.295776e-03)
(19, 9.295342e-03)
(20, 9.498435e-03)
(21, 9.386959e-03)
(22, 6.144740e-03)
(23, 6.121877e-03)
(24, 4.420142e-03)
(25, 4.419920e-03)
(26, 4.412728e-03)
(27, 4.383233e-03)
(28, 3.506470e-03)
(29, 3.421046e-03)
(30, 3.041117e-03)
(31, 3.029035e-03)
(32, 3.112418e-03)
(33, 3.095627e-03)
(34, 2.682227e-03)
(35, 2.672376e-03)
(36, 2.450298e-03)
(37, 2.443339e-03)
(38, 2.535180e-03)
(39, 2.522246e-03)
(40, 2.388772e-03)
(41, 2.383508e-03)
(42, 1.813668e-03)
(43, 1.811203e-03)
(44, 1.842867e-03)
(45, 1.835776e-03)
(46, 1.873238e-03)
(47, 1.864328e-03)
(48, 1.388333e-03)
(49, 1.385773e-03)
(50, 1.420012e-03)
(51, 1.416243e-03)
(52, 1.535459e-03)
(53, 1.529328e-03)
(54, 1.201197e-03)
(55, 1.200007e-03)
(56, 1.186495e-03)
(57, 1.185336e-03)
(58, 1.243536e-03)
(59, 1.241143e-03)
};
\addplot[color=blue, thick] coordinates {
(1, 1.448828e+00)
(2, 1.378150e+00)
(3, 1.273988e+00)
(4, 3.403911e-01)
(5, 3.443631e-01)
(6, 6.179796e-02)
(7, 6.204850e-02)
(8, 4.644909e-02)
(9, 4.520342e-02)
(10, 3.911347e-02)
(11, 3.840802e-02)
(12, 2.338273e-02)
(13, 2.336071e-02)
(14, 2.049472e-02)
(15, 2.018612e-02)
(16, 1.282321e-02)
(17, 1.276029e-02)
(18, 9.304274e-03)
(19, 9.303937e-03)
(20, 9.504507e-03)
(21, 9.393226e-03)
(22, 6.147984e-03)
(23, 6.124995e-03)
(24, 4.422308e-03)
(25, 4.422090e-03)
(26, 4.414409e-03)
(27, 4.384745e-03)
(28, 4.423891e-03)
(29, 4.367777e-03)
(30, 3.044762e-03)
(31, 3.032627e-03)
(32, 3.166795e-03)
(33, 3.146779e-03)
(34, 2.907428e-03)
(35, 2.897576e-03)
(36, 2.452472e-03)
(37, 2.445580e-03)
(38, 2.500900e-03)
(39, 2.489060e-03)
(40, 2.375560e-03)
(41, 2.370680e-03)
(42, 1.792343e-03)
(43, 1.787897e-03)
(44, 1.814140e-03)
(45, 1.807026e-03)
(46, 1.969272e-03)
(47, 1.960127e-03)
(48, 1.390052e-03)
(49, 1.387436e-03)
(50, 1.456683e-03)
(51, 1.452143e-03)
(52, 1.540023e-03)
(53, 1.533956e-03)
(54, 1.203400e-03)
(55, 1.202166e-03)
(56, 1.219305e-03)
(57, 1.218461e-03)
(58, 1.252456e-03)
(59, 1.250072e-03)
};
\addplot[color=blue, thick] coordinates {
(1, 1.448826e+00)
(2, 1.378157e+00)
(3, 1.273993e+00)
(4, 3.403942e-01)
(5, 3.443663e-01)
(6, 6.179921e-02)
(7, 6.204994e-02)
(8, 4.645094e-02)
(9, 4.520535e-02)
(10, 3.911566e-02)
(11, 3.841013e-02)
(12, 2.338405e-02)
(13, 2.336200e-02)
(14, 2.049559e-02)
(15, 2.018703e-02)
(16, 1.282419e-02)
(17, 1.276120e-02)
(18, 9.306400e-03)
(19, 9.306086e-03)
(20, 9.506021e-03)
(21, 9.394789e-03)
(22, 6.148799e-03)
(23, 6.125779e-03)
(24, 4.422836e-03)
(25, 4.422619e-03)
(26, 4.415000e-03)
(27, 4.385247e-03)
(28, 3.037915e-03)
(29, 3.029148e-03)
(30, 3.054532e-03)
(31, 3.042891e-03)
(32, 3.120249e-03)
(33, 3.103350e-03)
(34, 2.742305e-03)
(35, 2.732243e-03)
(36, 2.453023e-03)
(37, 2.446151e-03)
(38, 2.432931e-03)
(39, 2.425099e-03)
(40, 2.368051e-03)
(41, 2.362993e-03)
(42, 1.813671e-03)
(43, 1.806986e-03)
(44, 1.838392e-03)
(45, 1.831158e-03)
(46, 1.512797e-03)
(47, 1.512128e-03)
(48, 1.395253e-03)
(49, 1.392247e-03)
(50, 1.484158e-03)
(51, 1.478900e-03)
(52, 1.545876e-03)
(53, 1.539882e-03)
(54, 1.203942e-03)
(55, 1.202695e-03)
(56, 1.250609e-03)
(57, 1.248260e-03)
(58, 1.254371e-03)
(59, 1.252045e-03)
};
\end{axis}
\end{tikzpicture}%

%% file: wave_fd_limit_smooth4.tikz
\begin{tikzpicture}

\begin{axis}[
width=5.5cm, height=5.5cm, scale only axis,
xmin=0, xmax=60, ymode=log, ymin=5e-6, ymax=0.5e1,
ytick={1,1e-1,1e-2,1e-3,1e-4,1e-5,1e-6}, yminorticks=false,
yticklabels={$10^0$,$10^{-1}$,$10^{-2}$,$10^{-3}$,$10^{-4}$,$10^{-5}$,$10^{-6}$}
]
\addplot[color=blue, thick] coordinates {
(1, 1.336300e+00)
(2, 1.331207e+00)
(3, 1.217367e+00)
(4, 1.043765e+00)
(5, 1.037484e+00)
(6, 4.278215e-01)
(7, 4.255176e-01)
(8, 2.583208e-01)
(9, 2.469241e-01)
(10, 1.095694e-01)
(11, 1.095813e-01)
(12, 4.260329e-02)
(13, 4.089052e-02)
(14, 1.384022e-02)
(15, 1.378306e-02)
(16, 5.474194e-03)
(17, 5.310142e-03)
(18, 2.377840e-03)
(19, 2.357526e-03)
(20, 8.578573e-04)
(21, 8.469377e-04)
(22, 6.521312e-04)
(23, 6.395707e-04)
(24, 3.110742e-04)
(25, 3.109293e-04)
(26, 1.837078e-04)
(27, 1.800802e-04)
(28, 1.539534e-04)
(29, 1.522532e-04)
(30, 1.591434e-04)
(31, 1.590530e-04)
(32, 1.448107e-04)
(33, 1.434059e-04)
(34, 1.111395e-04)
(35, 1.085042e-04)
(36, 5.079898e-05)
(37, 5.072510e-05)
(38, 6.714737e-05)
(39, 6.700398e-05)
(40, 5.716243e-05)
(41, 5.646579e-05)
(42, 2.747329e-05)
(43, 2.748526e-05)
(44, 2.940517e-05)
(45, 2.943224e-05)
(46, 3.348732e-05)
(47, 3.316678e-05)
(48, 1.902918e-05)
(49, 1.902073e-05)
(50, 1.888606e-05)
(51, 1.887071e-05)
(52, 2.037516e-05)
(53, 2.023595e-05)
(54, 1.482131e-05)
(55, 1.480400e-05)
(56, 1.374790e-05)
(57, 1.372370e-05)
(58, 1.364122e-05)
(59, 1.360759e-05)
};
\addplot[color=blue, thick] coordinates {
(1, 1.334373e+00)
(2, 1.330235e+00)
(3, 1.215996e+00)
(4, 1.039159e+00)
(5, 1.033137e+00)
(6, 4.241398e-01)
(7, 4.216600e-01)
(8, 2.555592e-01)
(9, 2.443184e-01)
(10, 1.070531e-01)
(11, 1.070490e-01)
(12, 4.170398e-02)
(13, 4.002740e-02)
(14, 1.307973e-02)
(15, 1.303724e-02)
(16, 5.233133e-03)
(17, 5.066723e-03)
(18, 2.140099e-03)
(19, 2.125618e-03)
(20, 8.033600e-04)
(21, 7.867059e-04)
(22, 7.954523e-04)
(23, 7.899142e-04)
(24, 4.493266e-04)
(25, 4.393121e-04)
(26, 2.550831e-04)
(27, 2.548450e-04)
(28, 2.538635e-04)
(29, 2.529984e-04)
(30, 1.293042e-04)
(31, 1.269348e-04)
(32, 1.262842e-04)
(33, 1.256274e-04)
(34, 1.174292e-04)
(35, 1.161172e-04)
(36, 4.635217e-05)
(37, 4.582868e-05)
(38, 4.496104e-05)
(39, 4.482458e-05)
(40, 5.179611e-05)
(41, 5.129551e-05)
(42, 2.660047e-05)
(43, 2.634627e-05)
(44, 2.496168e-05)
(45, 2.496104e-05)
(46, 2.875069e-05)
(47, 2.855125e-05)
(48, 2.460481e-05)
(49, 2.426681e-05)
(50, 1.624637e-05)
(51, 1.624469e-05)
(52, 1.701134e-05)
(53, 1.690544e-05)
(54, 1.443749e-05)
(55, 1.424214e-05)
(56, 1.206627e-05)
(57, 1.200476e-05)
(58, 1.104844e-05)
(59, 1.104780e-05)
};
\addplot[color=blue, thick] 
coordinates {
(1, 1.336208e+00)
(2, 1.331161e+00)
(3, 1.217302e+00)
(4, 1.043546e+00)
(5, 1.037277e+00)
(6, 4.276458e-01)
(7, 4.253336e-01)
(8, 2.581898e-01)
(9, 2.468004e-01)
(10, 1.094485e-01)
(11, 1.094598e-01)
(12, 4.256058e-02)
(13, 4.084945e-02)
(14, 1.380327e-02)
(15, 1.374686e-02)
(16, 5.462283e-03)
(17, 5.298071e-03)
(18, 2.365900e-03)
(19, 2.345880e-03)
(20, 8.533176e-04)
(21, 8.422472e-04)
(22, 6.508554e-04)
(23, 6.383868e-04)
(24, 3.123946e-04)
(25, 3.119671e-04)
(26, 2.571873e-04)
(27, 2.541461e-04)
(28, 2.378369e-04)
(29, 2.339774e-04)
(30, 1.493942e-04)
(31, 1.477456e-04)
(32, 1.527239e-04)
(33, 1.514289e-04)
(34, 1.387904e-04)
(35, 1.371433e-04)
(36, 5.054556e-05)
(37, 5.046102e-05)
(38, 5.984188e-05)
(39, 6.025873e-05)
(40, 5.695098e-05)
(41, 5.626314e-05)
(42, 2.736751e-05)
(43, 2.731440e-05)
(44, 2.777590e-05)
(45, 2.779092e-05)
(46, 3.327346e-05)
(47, 3.295732e-05)
(48, 2.129427e-05)
(49, 2.116365e-05)
(50, 1.877857e-05)
(51, 1.876501e-05)
(52, 2.022915e-05)
(53, 2.008846e-05)
(54, 1.459544e-05)
(55, 1.453044e-05)
(56, 1.379374e-05)
(57, 1.377793e-05)
(58, 1.352216e-05)
(59, 1.348923e-05)
};
\addplot[color=darkgreen, dashdotted, thick] 
coordinates {%
(1, 1.328526e+00)
(2, 1.327243e+00)
(3, 1.211790e+00)
(4, 1.025038e+00)
(5, 1.019792e+00)
(6, 4.130311e-01)
(7, 4.099898e-01)
(8, 2.470401e-01)
(9, 2.362969e-01)
(10, 9.968953e-02)
(11, 9.960386e-02)
(12, 3.893493e-02)
(13, 3.738764e-02)
(14, 1.096517e-02)
(15, 1.095423e-02)
(16, 4.596250e-03)
(17, 4.435104e-03)
(18, 1.561467e-03)
(19, 1.559351e-03)
(20, 9.357898e-04)
(21, 9.087172e-04)
(22, 4.759306e-04)
(23, 4.749917e-04)
(24, 3.338877e-04)
(25, 3.263900e-04)
(26, 1.381169e-04)
(27, 1.376392e-04)
(28, 1.359102e-04)
(29, 1.358520e-04)
(30, 8.354790e-05)
(31, 8.207423e-05)
(32, 5.840213e-05)
(33, 5.782239e-05)
(34, 5.642071e-05)
(35, 5.636666e-05)
(36, 3.238877e-05)
(37, 3.191292e-05)
(38, 2.773278e-05)
(39, 2.731857e-05)
(40, 1.763001e-05)
(41, 1.754671e-05)
(42, 1.752664e-05)
(43, 1.750357e-05)
(44, 1.204407e-05)
(45, 1.195305e-05)
(46, 6.235337e-06)
(47, 6.132623e-06)
(48, 5.438071e-06)
(49, 5.353486e-06)
(50, 3.810492e-06)
(51, 3.808340e-06)
(52, 2.251542e-06)
(53, 2.235852e-06)
(54, 1.978386e-06)
(55, 1.957960e-06)
(56, 1.798598e-06)
(57, 1.766840e-06)
(58, 1.066237e-06)
(59, 1.045281e-06)
 };
 \addplot[color=red, dashed, thick] 
coordinates {%
(1, 1.305740e+00)
(2, 1.314814e+00)
(3, 1.194570e+00)
(4, 9.674019e-01)
(5, 9.650345e-01)
(6, 3.706346e-01)
(7, 3.650502e-01)
(8, 2.119771e-01)
(9, 2.034217e-01)
(10, 7.517378e-02)
(11, 7.439917e-02)
(12, 2.791685e-02)
(13, 2.698739e-02)
(14, 5.584992e-03)
(15, 5.541306e-03)
(16, 2.607230e-03)
(17, 2.548598e-03)
(18, 6.109826e-04)
(19, 5.874737e-04)
(20, 2.388686e-04)
(21, 2.386119e-04)
(22, 7.356210e-05)
(23, 7.124466e-05)
(24, 1.261803e-05)
(25, 1.239716e-05)
(26, 1.176295e-05)
(27, 1.164711e-05)
(28, 5.645409e-06)
(29, 5.602458e-06)
(30, 1.431301e-06)
(31, 1.386468e-06)
(32, 1.173761e-06)
(33, 1.129117e-06)
(34, 3.705930e-07)
(35, 3.609692e-07)
(36, 3.223554e-07)
(37, 3.203636e-07)
(38, 7.308222e-08)
(39, 7.242391e-08)
(40, 4.419332e-08)
(41, 4.299734e-08)
(42, 2.291215e-08)
(43, 2.219446e-08)
(44, 3.762880e-09)
(45, 3.682936e-09)
(46, 3.194078e-09)
(47, 3.171979e-09)
(48, 1.203015e-09)
(49, 1.200063e-09)
(50, 7.503180e-10)
(51, 7.392172e-10)
(52, 4.058227e-10)
(53, 3.947321e-10)
(54, 3.422987e-10)
(55, 3.334693e-10)
(56, 8.861925e-11)
(57, 8.502679e-11)
(58, 5.251587e-11)
(59, 5.087298e-11)
 };
\addplot[color=blue, thick] coordinates {
(1, 1.335841e+00)
(2, 1.330976e+00)
(3, 1.217041e+00)
(4, 1.042670e+00)
(5, 1.036450e+00)
(6, 4.269434e-01)
(7, 4.245980e-01)
(8, 2.576649e-01)
(9, 2.463050e-01)
(10, 1.089659e-01)
(11, 1.089745e-01)
(12, 4.238963e-02)
(13, 4.068515e-02)
(14, 1.365622e-02)
(15, 1.360275e-02)
(16, 5.415087e-03)
(17, 5.250280e-03)
(18, 2.318774e-03)
(19, 2.299913e-03)
(20, 8.369526e-04)
(21, 8.251432e-04)
(22, 6.616026e-04)
(23, 6.495758e-04)
(24, 3.678284e-04)
(25, 3.632679e-04)
(26, 2.942525e-04)
(27, 2.931800e-04)
(28, 2.701251e-04)
(29, 2.644438e-04)
(30, 1.441458e-04)
(31, 1.421692e-04)
(32, 1.454217e-04)
(33, 1.442594e-04)
(34, 1.351348e-04)
(35, 1.335711e-04)
(36, 4.957896e-05)
(37, 4.945990e-05)
(38, 5.505377e-05)
(39, 5.535443e-05)
(40, 5.604588e-05)
(41, 5.539174e-05)
(42, 2.737146e-05)
(43, 2.725041e-05)
(44, 2.720301e-05)
(45, 2.721849e-05)
(46, 3.240409e-05)
(47, 3.210784e-05)
(48, 1.879457e-05)
(49, 1.876281e-05)
(50, 1.832724e-05)
(51, 1.831960e-05)
(52, 1.963334e-05)
(53, 1.948903e-05)
(54, 1.392814e-05)
(55, 1.391659e-05)
(56, 1.325121e-05)
(57, 1.323594e-05)
(58, 1.316305e-05)
(59, 1.313699e-05)
};
\addplot[color=blue, thick] coordinates {
(1, 1.336323e+00)
(2, 1.331218e+00)
(3, 1.217383e+00)
(4, 1.043820e+00)
(5, 1.037535e+00)
(6, 4.278655e-01)
(7, 4.255636e-01)
(8, 2.583536e-01)
(9, 2.469551e-01)
(10, 1.095997e-01)
(11, 1.096117e-01)
(12, 4.261396e-02)
(13, 4.090079e-02)
(14, 1.384947e-02)
(15, 1.379212e-02)
(16, 5.477178e-03)
(17, 5.313168e-03)
(18, 2.380835e-03)
(19, 2.360447e-03)
(20, 8.590194e-04)
(21, 8.481352e-04)
(22, 6.527093e-04)
(23, 6.401257e-04)
(24, 3.114915e-04)
(25, 3.113680e-04)
(26, 1.766772e-04)
(27, 1.730446e-04)
(28, 1.507097e-04)
(29, 1.504034e-04)
(30, 1.628790e-04)
(31, 1.627771e-04)
(32, 1.210191e-04)
(33, 1.189372e-04)
(34, 8.686009e-05)
(35, 8.517983e-05)
(36, 5.113991e-05)
(37, 5.108113e-05)
(38, 6.356276e-05)
(39, 6.283353e-05)
(40, 5.715188e-05)
(41, 5.645047e-05)
(42, 2.749764e-05)
(43, 2.750945e-05)
(44, 2.885263e-05)
(45, 2.887938e-05)
(46, 3.353979e-05)
(47, 3.321824e-05)
(48, 1.886064e-05)
(49, 1.885919e-05)
(50, 1.892396e-05)
(51, 1.890763e-05)
(52, 2.041181e-05)
(53, 2.027301e-05)
(54, 1.453939e-05)
(55, 1.452873e-05)
(56, 1.384263e-05)
(57, 1.381618e-05)
(58, 1.364312e-05)
(59, 1.360917e-05)
};
\addplot[color=blue, thick] coordinates {
(1, 1.336329e+00)
(2, 1.331221e+00)
(3, 1.217387e+00)
(4, 1.043833e+00)
(5, 1.037548e+00)
(6, 4.278764e-01)
(7, 4.255751e-01)
(8, 2.583618e-01)
(9, 2.469628e-01)
(10, 1.096072e-01)
(11, 1.096193e-01)
(12, 4.261663e-02)
(13, 4.090335e-02)
(14, 1.385178e-02)
(15, 1.379439e-02)
(16, 5.477925e-03)
(17, 5.313924e-03)
(18, 2.381584e-03)
(19, 2.361178e-03)
(20, 8.593115e-04)
(21, 8.484362e-04)
(22, 6.528701e-04)
(23, 6.402807e-04)
(24, 3.116392e-04)
(25, 3.115182e-04)
(26, 1.765010e-04)
(27, 1.728564e-04)
(28, 1.668901e-04)
(29, 1.631895e-04)
(30, 1.627372e-04)
(31, 1.625356e-04)
(32, 6.179332e-05)
(33, 6.059674e-05)
(34, 5.357197e-05)
(35, 5.287493e-05)
(36, 5.524167e-05)
(37, 5.531723e-05)
(38, 6.274707e-05)
(39, 6.201224e-05)
(40, 5.681756e-05)
(41, 5.611304e-05)
(42, 2.749454e-05)
(43, 2.750588e-05)
(44, 2.853438e-05)
(45, 2.855798e-05)
(46, 3.355595e-05)
(47, 3.323446e-05)
(48, 2.438948e-05)
(49, 2.414893e-05)
(50, 1.891052e-05)
(51, 1.889498e-05)
(52, 2.042129e-05)
(53, 2.028261e-05)
(54, 1.408738e-05)
(55, 1.407719e-05)
(56, 1.375438e-05)
(57, 1.372802e-05)
(58, 1.362957e-05)
(59, 1.359323e-05)
};
\end{axis}
\end{tikzpicture}%

%% file: wave_fdmg_timings_smooth2.tikz
\begin{tikzpicture}

\begin{axis}[
width=5.5cm, height=5.5cm, scale only axis,
xmin=0, xmax=13, ymode=log, ymin=5e-6, ymax=0.5e1,
ytick={1,1e-1,1e-2,1e-3,1e-4,1e-5}, yminorticks=false
]
\addplot[blue, solid, thick] coordinates {
(0.003748, 1.448826e+00)
(0.084215, 1.378157e+00)
(0.176802, 1.273790e+00)
(0.274142, 3.403850e-01)
(0.376394, 3.443200e-01)
(0.482685, 6.179011e-02)
(0.594700, 6.204134e-02)
(0.710460, 4.644390e-02)
(0.831215, 4.520094e-02)
(0.956823, 3.911189e-02)
(1.087316, 3.840339e-02)
(1.221375, 2.338308e-02)
(1.360787, 2.336042e-02)
(1.504484, 2.049462e-02)
(1.654652, 2.018507e-02)
(1.805801, 1.282403e-02)
(1.964499, 1.275916e-02)
(2.126069, 9.303995e-03)
(2.295964, 9.303734e-03)
(2.475347, 9.483413e-03)
(2.651643, 9.445405e-03)
(2.830435, 9.330937e-03)
(3.016178, 9.224401e-03)
(3.206220, 9.230884e-03)
(3.403354, 9.144187e-03)
(3.600610, 6.139648e-03)
(3.805789, 6.103526e-03)
(4.014108, 4.387101e-03)
(4.229707, 4.385517e-03)
(4.450970, 4.332874e-03)
(4.673704, 4.320306e-03)
(4.937665, 4.287740e-03)
(5.142834, 4.255345e-03)
(5.374255, 3.015001e-03)
(5.617751, 3.002538e-03)
(5.863095, 3.096201e-03)
(6.114813, 3.077386e-03)
(6.367576, 3.186331e-03)
(6.625034, 3.167140e-03)
(6.893008, 2.577441e-03)
(7.166501, 2.577122e-03)
(7.439334, 2.425682e-03)
(7.716772, 2.419542e-03)
(8.001566, 2.494498e-03)
(8.252488, 2.486302e-03)
(8.537321, 2.499239e-03)
(8.841811, 2.491671e-03)
(9.131564, 2.160919e-03)
(9.448717, 2.160506e-03)
(9.751045, 1.800100e-03)
(10.075577, 1.793781e-03)
(10.407200, 1.831246e-03)
(10.717321, 1.823845e-03)
(11.035292, 1.418649e-03)
(11.371022, 1.415931e-03)
(11.711837, 1.384387e-03)
(12.047452, 1.380181e-03)
(12.386108, 1.524929e-03)
(12.754854, 1.520605e-03)
};
\addplot[color=red, dashed, thick] 
coordinates {
(0.073759, 1.053864e+00)
(0.221110, 8.807128e-01)
(0.368490, 7.641378e-01)
(0.674854, 5.766928e-01)
(1.303201, 3.828693e-01)
(1.918015, 2.849586e-01)
(2.288028, 2.265852e-01)
(2.781232, 1.879756e-01)
(3.756280, 1.317853e-01)
(4.993787, 1.014784e-01)
(6.221598, 8.253106e-02)
(7.435901, 6.956720e-02)
(9.700778, 5.297587e-02)
(11.532057, 4.280982e-02)
(13.898083, 3.338197e-02)
};
\end{axis}
\end{tikzpicture}%

%% file: wave_fdmg_timings_smooth4.tikz
\begin{tikzpicture}

\begin{axis}[
width=5.5cm, height=5.5cm, scale only axis,
xmin=0, xmax=13, ymode=log, ymin=5e-6, ymax=0.5e1,
ytick={1,1e-1,1e-2,1e-3,1e-4,1e-5}, yminorticks=false
]
\addplot[color=blue, thick] coordinates {
(0.004017, 1.336329e+00)
(0.088331, 1.331223e+00)
(0.184509, 1.217383e+00)
(0.284022, 1.043821e+00)
(0.397462, 1.037429e+00)
(0.510822, 4.278253e-01)
(0.622018, 4.255397e-01)
(0.720264, 2.583344e-01)
(0.834189, 2.469174e-01)
(0.958698, 1.095925e-01)
(1.088279, 1.096033e-01)
(1.222761, 4.260897e-02)
(1.361437, 4.089121e-02)
(1.514731, 1.385101e-02)
(1.661135, 1.379152e-02)
(1.832865, 5.476537e-03)
(1.979980, 5.312379e-03)
(2.149273, 2.408327e-03)
(2.316390, 2.390351e-03)
(2.489493, 2.337596e-03)
(2.668243, 2.306766e-03)
(2.856291, 1.348507e-03)
(3.043005, 1.285433e-03)
(3.234293, 8.617175e-04)
(3.428167, 8.614540e-04)
(3.631290, 5.868296e-04)
(3.834993, 5.736784e-04)
(4.050477, 3.159459e-04)
(4.263556, 3.160093e-04)
(4.480925, 3.205360e-04)
(4.701986, 3.197620e-04)
(4.928152, 1.533397e-04)
(5.159563, 1.501629e-04)
(5.405119, 1.484724e-04)
(5.641393, 1.471103e-04)
(5.900873, 1.700968e-04)
(6.146392, 1.697563e-04)
(6.393655, 1.496397e-04)
(6.651527, 1.478114e-04)
(6.924241, 9.878218e-05)
(7.195016, 9.624876e-05)
(7.469668, 5.148333e-05)
(7.747323, 5.155273e-05)
(8.033838, 6.674378e-05)
(8.317317, 6.666157e-05)
(8.625978, 6.400833e-05)
(8.919734, 6.322491e-05)
(9.317393, 4.429773e-05)
(9.559944, 4.381350e-05)
(9.868687, 3.204996e-05)
(10.182328, 3.199064e-05)
(10.509896, 3.123885e-05)
(10.839505, 3.109250e-05)
(11.180168, 3.348574e-05)
(11.511488, 3.319523e-05)
(11.858424, 2.582381e-05)
(12.208816, 2.580753e-05)
(12.548546, 2.073994e-05)
(12.911725, 2.064762e-05)
};
\addplot[color=red, dashed, thick] 
coordinates {
(0.072422, 1.056976e+00)
(0.218380, 9.828198e-01)
(0.352970, 9.053014e-01)
(0.638314, 7.485975e-01)
(1.263437, 5.516419e-01)
(1.801386, 4.356458e-01)
(2.257602, 3.595609e-01)
(2.705761, 3.059252e-01)
(3.817614, 2.226174e-01)
(4.900937, 1.748920e-01)
(6.092638, 1.440031e-01)
(7.102697, 1.223892e-01)
(9.343668, 9.414837e-02)
(11.212782, 7.651858e-02)
(13.679340, 6.446884e-02)
};
\end{axis}
\end{tikzpicture}%

%% file: Wave_error_coarse.tikz
%
%
\begin{tikzpicture}

\begin{axis}[
width=5.5cm, height=5.5cm, scale only axis,
xmin=0, xmax=10, ymode=log, ymin=1e-7, ymax=5e1,
ytick={1,1e-2,1e-4,1e-6}, yminorticks=false
]
\addplot[mark=square*, mark size=1.5pt, color=blue, thick] 
  table[row sep=crcr]{%
1	5.242755598125609\\
2	0.863377683927727\\
3	0.189181842401970\\
4	0.086165326284569\\
5	0.013130385678401\\
6	0.000695392510823\\
7	0.000185520704550\\
8	0.000016565317841\\
9	0.000003278479751\\
10	0.000000215121084\\
11	0.000000080078045\\
12	0.000000004216910\\
13	0.000000002300867\\
14	0.000000000899929\\
15	0.000000000567921\\
16	0.000000000618818\\
17	0.000000000257852\\
18	0.000000000136463\\
19	0.000000000082500\\
20	0.000000000083149\\
};
\addplot[mark=otimes*, mark size=1.5pt, color=red, thick] 
  table[row sep=crcr]{%
1	4.018265447900666\\
2	1.414486595632801\\
3	0.798862104551477\\
4	0.217273222616185\\
5	0.042351815957628\\
6	0.003300159404799\\
7	0.000845185594807\\
8	0.000100466136334\\
9	0.000021265986723\\
10	0.000001679510568\\
11	0.000000598908516\\
12	0.000000032039214\\
13	0.000000013928951\\
14	0.000000002767455\\
15	0.000000000489732\\
16	0.000000000479634\\
17	0.000000000337725\\
18	0.000000000346319\\
19	0.000000000143842\\
20	0.000000000116174\\
};
\end{axis}
\end{tikzpicture}%

%% file: Wave_error_fine.tikz
%
%
\begin{tikzpicture}

\begin{axis}[
width=5.5cm, height=5.5cm, scale only axis,
xmin=0, xmax=20, ymode=log, ymin=1e-7, ymax=5e1,
ytick={1,1e-2,1e-4,1e-6}, yminorticks=false
]
\addplot[mark=square*, mark size=1.5pt, color=blue, thick] 
  table[row sep=crcr]{%
1	20.17653259252523\\
2	3.055039497289574\\
3	0.560668277529479\\
4	0.760253514669822\\
5	0.253621107486996\\
6	0.080459028702713\\
7	0.036352100736870\\
8	0.014722182286621\\
9	0.012982497424705\\
10	0.004432598143705\\
11	0.003657759487273\\
12	0.001936461361585\\
13	0.000901647309547\\
14	0.000644084864813\\
15	0.000209171706896\\
16	0.000148222323800\\
17	0.000095717445254\\
18	0.000035933402022\\
19	0.000020040048595\\
20	0.000014007244769\\
};
\addplot[mark=otimes*, mark size=1.5pt, color=red, thick] 
  table[row sep=crcr]{%
1	15.53062224988302\\
2	5.394461921493565\\
3	7.335553274781717\\
4	3.836832637542929\\
5	2.478956688745302\\
6	1.266042892054227\\
7	0.398348449176463\\
8	0.248822590168347\\
9	0.167241345278061\\
10	0.074881151644660\\
11	0.052265399290188\\
12	0.028124061451837\\
13	0.018495151586854\\
14	0.014324199373596\\
15	0.005669554984814\\
16	0.003798619544331\\
17	0.002004907251579\\
18	0.000784886576884\\
19	0.000610146829280\\
20	0.000417643216015\\
};
\end{axis}
\end{tikzpicture}%